\documentclass{amsart}

\usepackage{amsmath, amssymb, amsthm, mathtools, amsfonts, mathrsfs}

\usepackage{graphicx, epsfig, verbatim, multirow}
\usepackage{enumerate, xcolor, color, enumitem, multirow, tabularx, array,makecell, diagbox}
\usepackage[colorlinks,allcolors=blue]{hyperref} 
\usepackage{cleveref}
\usepackage{appendix,pdfsync}
\usepackage{float}
\usepackage{setspace}
\usepackage{xpatch}
\usepackage{subcaption}
\usepackage{pst-node}
\usepackage{tikz-cd}
\usetikzlibrary{positioning}
\usetikzlibrary{shapes,arrows}
\usepackage{tikz}

\newtheorem{theorem}{Theorem}[section]
\newtheorem{lemma}[theorem]{Lemma}
\newtheorem{proposition}[theorem]{Proposition}
\newtheorem{conjecture}[theorem]{Conjecture}

\theoremstyle{definition}
\newtheorem{definition}[theorem]{Definition}

\theoremstyle{remark}
\newtheorem{remark}[theorem]{Remark}

\numberwithin{equation}{section}

\definecolor{blue1}{rgb}{0.10,0.60,0.8}
\definecolor{ss}{rgb}{0.16,0.01,0.40}
\definecolor{s}{rgb}{0.00,0.00,0.00}
\definecolor{p}{rgb}{0.00,0.00,0.79}
\definecolor{r}{rgb}{0.27,0.00,0.53}
\definecolor{t}{rgb}{0.22,0.00,0.44}
\definecolor{a}{rgb}{0.10,0.70,0.80}
\definecolor{an}{rgb}{0.60,0.00,1.00}

\def\os{<_{\cO}}
\newcommand{\D}[2]{\Delta_{#1}(#2)}
\newcommand{\di}[2]{#1\setminus\{#2\}}

\newcommand{\cD}{\mathcal{D}}

\newcommand{\cO}{\mathcal{O}}

\newcommand{\cS}{\mathcal{S}}
\newcommand{\T}{\mathcal{T}}
\newcommand{\Z}{\mathbb{Z}}

\newcommand{\al}{\alpha}
\newcommand{\bK}{\mathbb{K}}
\newcommand{\Cl}{\mathsf{Cl}}

\begin{document}

\title{Shellability of $3$-Cut Complexes of Squared Cycle Graphs}

\author{Pratiksha Chauhan}
 \address{School of Mathematical and Statistical Sciences, IIT Mandi, India}

 \email{d22037@students.iitmandi.ac.in}

 \author{Samir Shukla}
\address{School of Mathematical and Statistical  Sciences, IIT Mandi, India}
 \email{samir@iitmandi.ac.in}

 \author{Kumar Vinayak}
\address{Department of Mathematics, University of Kentucky, USA}

 \email{kumar.vinayak@uky.edu}

\subjclass[2020]{57M15, 52B22, 55U05, 05C69, 05E45}

 \keywords{Cut complex, Shellability, Squared cycle graphs, Homotopy}

\begin{abstract}

For a positive integer $k$, the $k$-cut complex of a graph $G$ is the simplicial complex whose facets are the $(|V(G)|-k)$-subsets $\sigma$ of the vertex set $V(G)$ of $G$ such that the induced subgraph of $G$ on $V(G) \setminus \sigma$ is disconnected.
These complexes first appeared in the master thesis of Denker and were further studied by Bayer et al.\ in [Topology of cut complexes of graphs, SIAM Journal on Discrete Mathematics, 38(2):1630-1675, 2024].   
In the same article, Bayer et al.\ conjectured that for $k \geq 3$, the $k$-cut complexes of squared cycle graphs are shellable. Moreover, they also conjectured about the Betti numbers of these complexes when $k=3$. In this article, we prove these conjectures for $k=3$.

\end{abstract}

\maketitle

\section{Introduction}
 All the graphs in this article are assumed to be finite and simple, that is, without loops and multiple edges. We denote the set of vertices and the set of edges of a graph $G$  by $V(G)$ and $E(G)$, respectively.
For an integer $k \geq 1$, the {\it $k$-cut complex} of a graph $G$, denoted as $\Delta_k(G)$, is the simplicial complex whose facets (maximal simplices) are $\sigma \subseteq  V(G)$  such that $|\sigma| = |V(G)|-k$ and the induced subgraph $G[V(G) \setminus \sigma]$ is disconnected.  These complexes first appeared in \cite{Denker2018} and were further studied by Bayer et al.\ in  \cite{Bayer2024Cutcomplex}. One of the main motivations behind cut complexes was a famous theorem 
of Ralf Fr{\"o}berg \cite{Froberg1990} connecting commutative algebra and graph theory through topology (see \Cref{thm:Fr}).

Let $\Delta$ be a simplicial complex on the vertex set $V(\Delta)= \{v_1, v_2, \ldots, v_n \}$ and let $\bK$ be a field. The Stanley-Reisner ideal  $I_{\Delta}$ of $\Delta$ is the ideal of the polynomial ring $\bK[x_1, \ldots, x_n]$ generated by the monomials corresponding to  minimal  subsets of $V(\Delta)$, which are not simplices  of $\Delta$, \textit{i.e.}, $I_{\Delta} = \langle x_{i_1}\ldots x_{i_k} : \{v_{i_1},v_{i_2}, \ldots, v_{i_k}\} \notin \Delta  \rangle$. The Stanley-Reisner ring  $\bK[\Delta]$ is the quotient ring  $\bK[x_1, \ldots, x_n]/I_{\Delta}$. For more details, we refer the reader to  \cite{Eagon1998}.

The Alexander dual $\Delta^{\vee}$ of the simplicial complex $\Delta$ is the simplicial complex on the vertex set $V(\Delta)$, whose simplices are the subsets of $V(\Delta)$ such that their complements are not simplices of $\Delta$, \textit{i.e.}, 
$$
\Delta^{\vee} = \{\sigma \subset V(\Delta): V(\Delta) \setminus \sigma \notin \Delta\}.
$$

The clique complex  $\mathsf{Cl}(G)$ of a graph $G$ is the simplicial complex whose simplices are $\sigma \subseteq V(G)$ such that the induced subgraph $G[\sigma]$ is a complete graph. It is easy to check that the  Stanley-Reisner ideal  $I_{\Cl(G)}$ of $\Cl(G)$,  is generated by quadratic square-free monomials.

\begin{theorem} [{\cite[Theorem 1]{Froberg1990}, \cite[p.\ 274]{Eagon1998}}] \label{thm:Fr}
	
	A Stanley–Reisner ideal $I_{\Delta}$ generated by quadratic square-free monomials has a $2$-linear resolution if and only if $\Delta$ is the clique complex $\Cl(G)$ of a chordal graph $G$.
	\end{theorem}

\begin{theorem}[{\cite[Proposition 8]{Eagon1998}}] \label{theorem:equivalence} 
	The following are equivalent for a graph $G$. 
	\begin{enumerate}
		\item  $G$ is chordal.
		
		\item  $\Cl(G)^{\vee}$  is  Cohen-Macaulay over a field $k$.
		
		\item  $\Cl(G)^{\vee}$  is  vertex decomposable.
	
		\end{enumerate} 
	\end{theorem}

For a pure simplicial complex, it is well known (see  \cite[Section 11]{BjornerTopologicalMethods}) that 
	$$
	\text{vertex decomposable}  \implies \text{shellable} \implies   \text{Cohen-Macaulay}.
	$$

Observe that  $\Cl(G)^{\vee} = \Delta_2(G)$ for any graph $G$.  Therefore,  \Cref{theorem:equivalence} implies the following.
\begin{align}\label{equation:equivalence}
G  \ \text{is chordal} \Longleftrightarrow  \Delta_2(G) \  \text{is shellable} \Longleftrightarrow  \Delta_2(G) \  \text{is vertex decomposable}.
\end{align}

For definitions of vertex decomposable, Cohen-Macaulay, and shellable complexes, see \Cref{section:prel}.  For more details on these concepts,  we refer to \cite[Section 11]{BjornerTopologicalMethods} and    \cite[Section 3.6]{JonssonBook}.

Inspired from  Fr{\"o}berg's theorem (\Cref{thm:Fr})  and \eqref{equation:equivalence}, Denker  \cite{Denker2018} introduced $k$-cut complexes of graphs which are a generalization of  $\Delta_2(G) = \Cl(G)^{\vee}$.  In  \cite{Bayer2024TotalCutcomplex}, Bayer et al.\ introduced another generalization of $\Delta_2(G)$, the total $k$-cut complexes.  For $k \geq 1$, the {\it total $k$-cut complex} of a graph $G$, denoted as $\Delta_k^t(G)$, is the simplicial complex whose facets are $\sigma \subseteq V(G)$ such that  $|\sigma| = |V(G)|-k$ and the induced subgraph $G[V(G) \setminus \sigma]$ does not contain any edge.  In particular, $\Delta_2(G) = \Delta_2^t(G)$ for any graph $G$.  

 For a graph $G$ and a positive integer $p$, the $p$-th power graph of $G$ is a graph $G^p$ where the set of vertices $V(G^p)=V(G)$ and any $\{u, v\}$ is an edge in $G^p$ if and only if there exists a path between $u$ and $v$ of length at most $p$  in $G$ (here, the length of a path is the number of edges in the path). Clearly $G^1 = G$. 
For $n \geq 3$, the cycle graph $C_n$ on $n$ vertices is a graph where the set of vertices $V(C_n) = \{0, 1, 2, \ldots, n-1\}$ and the set of edges $E(C_n) = \{\{i, i+1 \ (\text{mod} \ n)\} : 0 \leq i \leq n-1 \}$.
Observe that the $p$-th powered cycle graph $C_n^p$ on $n$ vertices is a graph with $V(C_n^p) = V(C_n)$ and the set of edges $E(C_n^p) =\{ \{i,i+j\ (\text{mod} \ n)\}  : 0\leq i\leq n-1 \text{ and } 1\leq j\leq p$\}. For $n \geq 2p+1$,  powered cycle graphs  $C_n^p$ are also $2p$-regular circulant graphs and hence Cayley graphs of the cyclic group on $n$ elements (for the definition of circulant graphs see \Cref{section:prel}).

It can be easily observed that for a graph $G$, $\Delta_1(G)$ is void. In \cite{Bayer2024TotalCutcomplex}, authors studied the total $2$-cut complexes of the cycle graphs $C_n$ \cite[Theorem 4.15]{Bayer2024TotalCutcomplex} and the squared cycle graphs $C_n^2$ \cite[Proposition 4.19]{Bayer2024TotalCutcomplex}. They proved that these complexes (if nonvoid) are homotopy equivalent to wedges of spheres of dimension lower than the dimension of $\Delta_2^t(C_n^p) = \Delta_2(C_n^p)$ for $p=1,2$. Therefore, the complexes  $\Delta_2(C_n)$ and $\Delta_2(C_n^2)$ are not shellable.
 In \cite[Proposition 7.11]{Bayer2024Cutcomplex}, authors proved that for $k\geq 3$, $\Delta_k(C_n)$ is shellable for $n\geq k+1$.
 Further,  based on their observations supported by Sage calculations (for $n \leq 13$ and $k \leq 5$),  they conjectured that for $k \geq 3$, $\Delta_k(C_n^2)$ are shellable for $n \geq k+6$.
  
  \begin{conjecture}[{\cite[Conjecture 7.25]{Bayer2024Cutcomplex}}]\label{conjecture}
  	For $k\ge 3$, the cut complex $\Delta_k(C_n^2)$ is shellable for $n\ge k+6$. For $k=3$ and $n\ge 9,$ the Betti numbers are 
  	$\binom{n-4}{2}-9 = \{1, 6, 12, 19, 27, \dots\}$.
  
  	\end{conjecture}

In this article, we prove \Cref{conjecture} for $k=3$. More precisely, the main result of this article is the following.

\begin{theorem}\label{theorem:main}
	Let $n \geq 9$. Then the $3$-cut complex $\D{3}{C_n^2}$ of a squared cycle graph is shellable. Moreover,
 $$
 \Delta_3(C_n^2) \simeq \bigvee\limits_{\binom{n-4}{2}-9} \mathbb{S}^{n-4}.
 $$
	\end{theorem}

Note that the sequence  $\binom{n-4}{2}-9 = \{1, 6, 12, 19, 27, \dots\}$ is OEIS A051936.

 We continued our investigation to analyze the $k$-cut complex $\Delta_k(C_n^p)$ for powered cycle graphs $C_n^p$ (for small values of $k$, $p$ and $n$) and computed their homology groups (with coefficient $\Z$) using SageMath (see Tables \ref{tab:2-cut}-\ref{tab:6-cut} in \Cref{section:future_directions}). 
By observing \Cref{tab:2-cut} and the fact that the complexes $\Delta_{2}(C_n)$ and $\Delta_2(C_n^2)$ (if nonvoid) are not shellable, we conjecture that if $\Delta_2(C_n^p)$ is nonvoid, then it is not shellable for $p \geq 3$ (\Cref{conjecture:2-cut}).

 It is proved in \cite[Proposition 7.11]{Bayer2024Cutcomplex} that $\Delta_3(C_n)$ is shellable for $n \geq 5 = 4+1$ (for $n\leq 4$, $\Delta_3(C_n)$ is void). In \Cref{theorem:main}, we have proved that $\Delta_3(C_n^2)$ is shellable for $n \geq 9 = 4.2 +1$.  Therefore, for $p \in \{1, 2\}$, $\Delta_3(C_n^p)$ is shellable for $n \geq 4p+1$. Based on these results and \Cref{tab:3-cut}, we conjecture that  $\Delta_3(C_n^p)$ is shellable for all $p \geq 3$ and  $n \geq 4p+1$ (\Cref{conjecture:3-cut powered cycle}).

 Tables \ref{tab:4-cut}, \ref{tab:5-cut} and \ref{tab:6-cut} bring up the question of whether the complexes $\Delta_k(C_n^p)$ (if nonvoid) are not shellable for $k\geq 4$ and $p\geq 3$. We also make a conjecture about the homotopy types of the $k$-cut complexes $\Delta_k(C_n^2)$ of squared cycles (\Cref{conjecture:k-cut squared cycle}).

This paper is organized as follows: In \Cref{section:prel}, we provide the necessary preliminaries related to graph theory and simplicial complexes. In \Cref{section:mainproof}, we give the proof of  \Cref{theorem:main}. This section is divided into two subsections.  To prove \Cref{theorem:main}, 
we define an order in the facets of $\Delta_3(C_n^2)$ and in \Cref{subsection:shelling order}, we show that this order is a shelling order. In \Cref{subsection:spanning facets}, we characterize and count the number of spanning facets for the shelling order to conclude the number of spheres appearing in the wedge in the homotopy type of $\Delta_3(C_n^2)$. Finally, in \Cref{section:future_directions}, based on our SageMath data, we propose several conjectures and questions about the shellability and homotopy types of the complexes $\Delta_k(C_n^p)$.   

\section{Preliminaries} \label{section:prel}

 In this section, we recall some basic definitions and results used in this article. 
\subsection{Graph}
A  \textit{graph} $G$ is a  pair $(V(G), E(G))$,  where $V(G)$  is the set of vertices
of $G$  and $E(G) \subseteq \binom{V(G)}{2}$
denotes the set of edges.
 If $\{x, y\} \in E(G)$, it is also denoted by $x \sim y$ and we say that $x$ is adjacent to $y$.  
 A {\it subgraph} $H$ of $G$ is a graph with $V(H) \subseteq V(G)$ and $E(H) \subseteq E(G)$.
For a subset $U \subseteq V(G)$, the \textit{induced subgraph} $G[U]$ is the subgraph whose set of vertices is $V(G[U]) = U$
and the set of edges is
 $E(G[U]) = \{\{a, b\} \in E(G) \ | \ a, b \in U\}$.

 For $x, y \in V(G)$, a {\it path} from $x$ to $y$ is  a sequence $x v_0 \ldots v_n y$ of vertices of $G$ such that $x \sim v_0, v_n \sim y$ and $v_i \sim v_{i+1}$ for all $0 \leq i\leq n-1$.  The number of edges in a path is the {\it length} of the path. 
 
 A graph is \textit{connected} if there exists a path between each pair of vertices. A graph is \textit{disconnected} if it is not connected. The graph with the empty set $\emptyset$ as its set of vertices is considered connected.
 
 Let $n \geq 2$ be a positive integer and $S \subset \{ 1, 2, \ldots, n-1\}$. 
 The {\it circulant graph} $C_n(S)$ is the graph whose set of vertices is $V(C_n(S)) =  \{0, 1, 2, \ldots, n-1\}$  and any two vertices $x$ and $y$ are
 adjacent if and only if $x-y \ (\text{mod \ $n$}) \in S \cup -S$, where $-S = \{n-a \ | \ a \in S\}$.
  Circulant 
 graphs are also Cayley graphs of $\mathbb{Z}_n$, the cyclic group on 
 $n$ elements.

We refer the reader to  \cite{bondy1976graph} and \cite{west} for more details about the graphs.

\subsection{Simplicial complex}
A {\it finite abstract simplicial complex} $\Delta$ is a collection of finite sets such that if $\tau \in \Delta$ and $\sigma \subset \tau$, then $\sigma \in \Delta$. 
The elements  of $\Delta$ are called {\it simplices} of $\Delta$. If $\sigma \subset \tau$, we say that $\sigma$ is a {\it face} of $\tau$.  The dimension of a simplex $\sigma$ is equal to $|\sigma| - 1$. 
The dimension of an abstract simplicial complex is the maximum of the dimensions of its simplices. The $0$-dimensional simplices are called \textit{vertices} of $\Delta$, and the set of vertices of $\Delta$ is denoted by $V(\Delta)$. If a simplex has dimension $d$, it is said to be $d$-{\it dimensional}. 
An abstract simplicial complex which is an empty collection of sets is called the \textit{void} abstract simplicial complex, and is denoted by $\emptyset$.

The {\it boundary} of a $d$-dimensional simplex $\sigma $ is the simplicial complex, consisting of  all  faces of $\sigma$ of dimension $\leq d-1$ and it is denoted by $Bd(\sigma).$
A simplex that is not a face of any other
 simplex is called a  {\it maximal simplex} or \textit{facet}. The set of maximal simplices of $\Delta$ is denoted by $M(\Delta)$. A simplicial complex is called
 {\it pure $d$-dimensional}, if all of its maximal simplices are of dimension $d$.  A \textit{subcomplex} $\Delta'$ of $\Delta$ is a simplicial complex such that $\sigma\in\Delta'$ implies $\sigma\in\Delta$.

 In this article, we consider any simplicial complex as a topological space, namely, its geometric realization. For the definition of geometric realization, we refer to the book \cite{Kozlov2008} by Kozlov. For terminology of algebraic topology used in
this article, see \cite{hatcher2005algebraic}.

Let $\Delta$ be a simplicial complex and let  $\tau$ be a simplex of $\Delta$.
The \textit{link} of $\tau$ is the   simplicial complex   defined as 
$$lk_{\Delta}(\tau) := \{\sigma\in \Delta \ | \ \sigma\cap \tau = \emptyset, \text{ and } \sigma \cup \tau \in \Delta\}.$$

    The \textit{deletion} of $\tau$ is the  simplicial complex 
 defined as 
$$dl_{\Delta}(\tau) := \{\sigma \in \Delta \ | \ \tau \not\subseteq\sigma\}.$$ 

\begin{definition}
A pure simplicial complex $\Delta$ is called \textit{vertex decomposable} if it is empty or if there exists a vertex $v$ of $\Delta$ such that both
 $lk_\Delta(v)$ and  $dl_\Delta(v)$ are vertex decomposable.
\end{definition}

\begin{definition}
    A complex $\Delta$ is \textit{$p$-acyclic} over a field $\mathbb{K}$ if the reduced homology group $\tilde{H}_i(\Delta;\mathbb{K})$ vanishes for $i \leq p$.
\end{definition}

\begin{definition}
    Let $\Delta$ be a pure simplicial complex and $\mathbb{K}$ be a field. Then $\Delta$ is \textit{Cohen-Macaulay} over $\mathbb{K}$ if
$lk_{\Delta}(\sigma)$ is $(dim\ lk_{\Delta}(\sigma) - 1)$-acyclic for each $\sigma$ in $\Delta$.    
\end{definition}

\subsection{Shellability}

\begin{definition} \cite[Section 12.1]{Kozlov2008} A simplicial complex $\Delta$ is called \textit{shellable} if its facets can be arranged in a linear order $F_1, F_2, \ldots, F_t$ in such a way that for all $2\leq j\leq t$, the subcomplex $(\bigcup_{i=1}^{j-1} F_i) \cap  F_j$ is pure and of dimension $|F_j|-2$. 
   
In other words, a simplicial complex $\Delta$ has a \textit{shelling order} $F_1, F_2, \ldots, F_t$ of its facets if and only if for any $i, j$ satisfying $1 \leq i < j \leq t$, there exists $1\leq r < j$ such that  $F_r \cap F_j = F_j\setminus\{\lambda\}$ for some $\lambda\in F_j\setminus F_i$. 
\end{definition}

A facet $F_k$ $(1< k \leq t)$  is called {\it spanning} with respect to the given shelling
order if $Bd(F_k) \subseteq \bigcup\limits_{i = 1}^{k-1} F_i$.  It is easy to check that $F_k$ is  a spanning facet  if for each $\lambda\in F_k$, there exists $r < k$ such that $F_r \cap F_k = F_k \setminus\{\lambda\}$.



\begin{theorem}\cite[Theorem 12.3]{Kozlov2008}\label{theorem:wedge}
Let $\Delta$ be a pure shellable simplicial complex of dimension $d$. Then  $\Delta$ has the homotopy type of a wedge of $\beta$ spheres of dimension  $d$, where $\beta$ is the number of total spanning facets in a given shelling. Hence $$\Delta \simeq \bigvee\limits_{\beta} \mathbb{S}^{d}.$$ 
\end{theorem}

\section{Proof of \Cref{theorem:main}}\label{section:mainproof}

 In this section, we prove the main result (\Cref{theorem:main}) of this article. Throughout the section, we fix    $n\geq 9$. Our aim is to define an order $\prec$ in the facets of $\Delta_3(C_n^2)$, which provides a shelling order. 

   Let
 
    \[ m:= \begin{cases*}
                    \frac{n+1}{2} & if $n$ is odd,   \\
                     \phantom{}\frac{n}{2} & if $n$ is even.
                 \end{cases*} \]

We arrange the elements of $V(C_n^2)=\{0,1,2,\ldots,n-1\}$ into an ordered set  $\mathcal{O} :=(\al_1,\al_2,\ldots,\al_n)$, where $\al_t  = m + (-1)^{t-1} \lfloor (t/2) \rfloor \ ( \text{mod} \ n)$ $\forall$ $1 \leq t \leq n$ and $\lfloor (t/2) \rfloor$ denotes the greatest integer less than or equal to $t/2$.   Hence \[ \cO= \begin{cases*}
                (m, m-1, m+1, m-2, m+2, \dots,n-2,2,n-1,1, 0), &if $n$ is odd,   \\
                     (m, m-1, m+1, m-2, m+2,\dots,2,n-2,1,n-1, 0), & if $n$ is even.
                 \end{cases*} \]  
                 
    We define an order $<_{\cO}$ in the elements  of $V(C_n^2)$ as follows: For any $x,y\in V(C_n^2)$, we say that $x<_{\cO} y$ if and only if
    $x=\al_s$ and $y=\al_t$ for some $1\leq s < t\leq n$. 

\begin{remark}\label{remark:order in S}
Let $x,y\in V(C_n^2)$.
   \begin{enumerate}[label=(\roman*)]
       \item If $x<m$, then $x<_{\cO} y$  if and only if either $y<x$ or $y\geq 2m-x$. \label{part:i}
        \item If $x\geq m$, then $x<_{\cO} y$  if and only if either $y< 2m-x$ or $y>x$. \label{part:ii}
    \item If $y<m$, then $x<_{\cO} y$  if and only if $y<x< 2m-y$. \label{part:iii}
     \item    If $y> m$, then $x<_{\cO} y$  if and only if $2m-y\leq x<y$. \label{part:iv}
   \end{enumerate}
        
\end{remark}   

 Throughout this section, we denote the complement of a set  $X \subseteq V(C_n^2)$ by $X^c$. 
 
 Let $\T$ be the set of all $(n-3)$-subsets of $V(C_n^2)$. We partition $\T$ into disjoint sets $\T_s$ as follows:
Let $\T_1 := \{\tau \in \mathcal{T}: \al_1\in \tau^c \}.$  For $s>1$, define
\[\T_{s} := \{ \tau \in \mathcal{T}: \al_s\in  \tau^c  \ \text{and } \al_1,\al_2, \dots,\al_{s-1}\notin\tau^c \}.\]  Clearly $\mathcal{T}= \bigsqcup \T_s$. 

 Throughout the section, if $\tau \in \T$ such that $\tau \in \T_s$ for some $s\geq 1,$ and $\tau^c = \{\al_s\} \sqcup \{s_1, s_2\}$, then we assume that $s_1 < s_2$.

 \begin{remark}\label{remark:Aj}
     Let $\tau \in \T$ be such that $\tau \in \T_s$ for some $s\geq 1,$ and ${\tau}^c = \{\al_s\} \sqcup \{s_1, s_2\}$. Observe that $\al_s\os s_1, s_2$ by definition of $\T_s$.
 \end{remark}  

Let $M(\Delta_3(C_n^2))$ be the set of facets of $\Delta_3(C_n^2)$. For any $F \in M(\Delta_3(C_n^2))$, we have $|F| =  n-3$ and $C_n^2[F^c]$ is disconnected. Clearly, $M(\Delta_3(C_n^2)) \subset \T=\bigsqcup \T_s$.

\begin{remark}\label{remark:range of al_s}
     Let $F \in M(\D{3}{C_n^2})$ such that $F\in\T_s$ and $F^c=\{\al_s\}\sqcup\{s_1,s_2\}$. Observe that since $C_n^2[F ^c]$ is disconnected and $\al_s\os s_1,s_2$, we get $\al_s\notin\{0,1,n-1\}$. Hence $2\leq\al_s\leq n-2$.
\end{remark}

\begin{definition}\label{definition:order_ll}
 Define an order $\ll$ in the elements of $M(\Delta_3(C_n^2))$ as follows: Let $F, F' \in M(\Delta_3(C_n^2))$. Let $F \in \T_s$ and $F' \in \T_t$ such that $F^c = \{\al_s\} \sqcup \{s_1, s_2\}$ and $F'^c = \{\al_t\} \sqcup \{t_1, t_2\}$. Then 
		$F \ll F'$ if and only if any  one of the following conditions is true: 
   \begin{enumerate}[label=(\roman*)]
     \item \label{s=t} If $s=t$ \textit{i.e.}, $\al_s=\al_t$, then either $s_1 < t_1$ or, $s_1 = t_1$ and $s_2 < t_2$.
     \item \label{s<t} If $s\neq t$ \textit{i.e.}, $\al_s\neq\al_t$, then $s<t$ \textit{i.e.}, $\al_s\os\al_t$.
   \end{enumerate} 
\end{definition}

    It is immediate that \Cref{definition:order_ll} defines a total order on $M(\Delta_3(C_n^2))$.
   
 To get our desired shelling order  $\prec$, we slightly modify the order $\ll$ by ``displacing" some elements of the poset $(M(\D{3}{C_n^2}), \ll)$ and get a new poset $(M(\D{3}{C_n^2}), $ $ \prec)$. For this, we first define a subset $\cD$ of $M(\Delta_3(C_n^2))$ as follows: A facet $F \in \cD$  if  and only if there exists  $s$ such that $F\in\T_s$ and $F$ satisfies any of the following conditions: 

\begin{enumerate}[label=$(\mathcal{D}_{\arabic*})$]
    
     \item\label{d1} ${F}^c = \{\al_s\}\sqcup\{\al_s-3,\al_s+1\}$, where  $\al_s=m+1$.

    \item\label{d2} ${F}^c = \{\al_s\}\sqcup\{\al_s-1,\al_s+3\}$, where $\al_s=m-1$.

    \item\label{d3} ${F}^c = \{\al_s\}\sqcup\{\al_s-4, \al_s-3\}$, where  $\al_s=m+1$. 
    
    \item\label{d4} 
    ${F}^c = \{\al_s\}\sqcup\{\al_s+3,\al_s+4\}$, where $\al_s\in\{m-1, m, \dots, n-6\}$.
      
\end{enumerate}

\begin{definition}\label{def:prec}
 Let $F,F'\in M(\D{3}{C_n^2})$.    The order $\prec$ on $M(\Delta_3(C_n^2))$ is defined as follows: We say $F \prec F'$ if and only if any  one of the following conditions is true:
\begin{enumerate}[label=(\roman*)]
    \item \label{def 1} $F, F' \notin \cD$ and $F \ll F'$.
    \item \label{def 2} $F\in\cD$, $F' \notin \cD$ such that  $F\in\T_s$, $F'\in\T_t$ for some $s<t$.
    \item \label{def 3} $F \notin \cD$,  $F' \in \cD$ such that $F\in\T_s$, $F'\in\T_t$, and either $s=t$ or  $s<t$ for some $s$ and $t$.
    \item \label{def 4} $F, F' \in \cD$ and $F \ll F'$.  
\end{enumerate}
\end{definition}

It can be easily observed that $\prec$ is a total order.



 \begin{proposition}\label{proposition:precedes}
    Let $F, F' \in M(\Delta_3(C_n^2))$. Let  $F \in \T_s$ and $F' \in \T_t$ for some $s \neq t$ such that  ${F}^c = \{\al_s\} \sqcup \{s_1, s_2\}$ and ${F'}^c = ({F}^c \setminus \{\nu\}) \sqcup \{\lambda\}$ for some $\nu \in {F}^c$ and $\lambda \in F$.  If   $\lambda<_{\cO} \alpha_s$,  then $F' \prec F$. 
    
 \end{proposition}

\begin{proof}
    

      Since $\al_s\os s_1,s_2$ by \Cref{remark:Aj}, $\lambda\os\al_s$ implies that $\lambda\os s_1,s_2$. Therefore $F'^c=\{\al_t\}\sqcup\{t_1,t_2\}$, where $\al_t=\lambda$ such that $t<s$ and $t_1,t_2\in F^c$. Hence $F' \prec F$ by \Cref{def:prec}.  
 \end{proof}
 
To prove \Cref{theorem:main},  in \Cref{subsection:shelling order}, we first show that the order $\prec$ given in \Cref{def:prec} provides a shelling order for $\D{3}{C_n^2}$. Later, in \Cref{subsection:spanning facets}, we identify and count all the spanning facets for this shelling order.

\subsection{Shelling Order}\label{subsection:shelling order}
In this section, we prove that the order $\prec$ given in \Cref{def:prec}  provides a shelling order for $\D{3}{C_n^2}$. To demonstrate that $\prec$  provides a shelling order for $\D{3}{C_n^2}$,  by definition of shellability, we need to prove the following: For  any $F_i, F_j\in M(\D{3}{C_n^2})$ with $F_i \prec F_j$, there is a facet $ F_r$  such that

 \begin{equation}  \label{eqn:super}
         F_r \prec F_j \text{ and } F_r \cap F_j = F_j\setminus\{\lambda\} \text{ for some } \lambda\in F_j\setminus F_i. \hfill \tag{\raisebox{-0.35ex}{\huge$\ast$}}
\end{equation}

\begin{proposition}\label{proposition:hash}
     Let $F_i \prec F_j$ be two elements of $M(\D{3}{C_n^2})$.  If there is an $F \in M(\D{3}{C_n^2})$ such that $F \prec F_j$ and $F^c=(\di{F_j^c}{\nu})\sqcup\{\lambda\}$ for some $\nu\in F_j^c$ and $\lambda\in F_j\setminus F_i$, then $F_r=F$ satisfies \eqref{eqn:super} for the pair $F_i \prec F_j$.     
 \end{proposition}
 
\begin{proof} 

    We already have $F \prec F_j$. Since $F^c=(\di{F_j^c}{\nu})\sqcup\{\lambda\}$, we get $F=(\di{F_j}{\lambda})\sqcup\{\nu\}$ and thus,  $F \cap F_j = F_j\setminus\{\lambda\}$, where $\lambda\in F_j\setminus F_i$. Hence $F_r=F$ satisfies \eqref{eqn:super} for the pair $F_i \prec F_j$.
 \end{proof}

Let  $F_i, F_j \in M(\Delta_3(C_n^2))$ such that $F_i \prec F_j$. Let   $F_i \in \T_{i_0}$ and $ F_j \in \T_{j_0}$, where   ${F_i}^c = \{\al_{i_0}\}\sqcup\{i_1,i_2\}$ and  ${F_j}^c = \{\al_{j_0}\}\sqcup\{j_1,j_2\}$. By our assumption, we have $i_1<i_2$ and $j_1<j_2$. For convenience, we drop the subscripts $i_0$ and $j_0$ from $\al_{i_0}$ and $\al_{j_0}$, and refer to them simply as $\al_i$ and $\al_j$, respectively. Our aim is to find an $F_r$ which satisfies \eqref{eqn:super} for the pair $F_i\prec F_j$.

We deal with the cases $F_j \in \cD$ and $F_j \notin \cD$ separately. The case $F_j \in \cD$ is considered in \Cref{case:D}, and the case  $F_j \notin \cD$  is considered in Lemmas \ref{case:A} and \ref{case:B}. In \Cref{case:A}, we deal with the case $i_0 = j_0$, and in \Cref{case:B}, we deal with the case $i_0 \neq j_0$. 
 Recall that for any $x,y\in V(C_n^2)$, $x \sim y$ denotes that $x$ is adjacent to $y$ in $C_n^2$ \textit{i.e.}, $\{x,y\}\in E(C_n^2)$.  

\begin{lemma}\label{case:A}
    Suppose that $F_j\notin\cD$. If $F_i\in\T_{j_0}$, then there exists an $F_r\in M(\Delta_3(C_n^2))$ that satisfies \eqref{eqn:super} for the pair $F_i\prec F_j$.
\end{lemma}

\begin{proof}
 Since $F_i\in\T_{j_0}$, we have $\al_i=\al_j$, and thus ${F_i}^c = \{\al_j\}\sqcup\{i_1,i_2\}$
    and ${F_j}^c = \{\al_j\}\sqcup\{j_1,j_2\}$.  Therefore, $\al_j\os i_1,i_2,j_1,j_2$ by \Cref{remark:Aj}. If $\{i_1, i_2\} \cap \{j_1,j_2\} \neq \emptyset$,  then \eqref{eqn:super} is satisfied by taking $F_r=F_i$. So, we assume that $\{i_1, i_2\} \cap \{j_1,j_2\} = \emptyset$.
    
    Since $F_i \prec F_j$, $F_j \notin \cD$ and $F_i, F_j \in \T_{j_0}$, using \Cref{def:prec}, we get $F_i \notin \cD$ and $F_i \ll F_j$. Therefore $F_i, F_j \in \T_{j_0}$ implies that either $i_1 < j_1$, or $i_1 = j_1$ and $i_2 < j_2$  by \Cref{definition:order_ll} \ref{s=t}. Since $i_1, i_2, j_1$, and $j_2$ are distinct, we get $i_1 < j_1 < j_2$. Moreover, $i_1 < i_2$. By \Cref{remark:range of al_s}, we have $2\leq\al_j\leq n-2$.     

     We now consider the following cases: (I) $j_1, j_2 < \al_j$, (II) $j_1, j_2 >\al_j$ and  (III) $j_1 < \al_j$ and $j_2 >\al_j$.

    \begin{enumerate}[label=(\MakeUppercase{\roman*})]
    \item $j_1, j_2 < \al_j$.
    
    Let ${F'}^c :=\{\al_j\}\sqcup\{i_1,j_2\}$. 
   Since $i_1<j_1<j_2<\al_j$, we have $i_1<\al_j-2$.  Suppose $i_1 \sim \alpha_j$.  Then  $\al_j\leq n-2$ implies that $i_1=0$ and $\al_j=n-2$.  Since  $C_n^2[{F_i}^c]$ is disconnected, we get $i_2\neq n-1$. Therefore, $i_1,i_2,j_1,j_2<\al_j$.  We have $\al_j=n-2>m$ (as $n\geq 9$) and $\al_j\os i_1,i_2,j_1,j_2$. Thus, by \Cref{remark:order in S} \ref{part:ii}, $ i_1,i_2,j_1,j_2 < 2m-\al_j = 2m-n+2\leq 3$, a contradiction as $i_1,i_2,j_1,j_2$ are distinct. Hence $i_1\nsim\al_j.$  Since $C_n^2[{F_j}^c]$ is disconnected, we have $j_2\nsim\al_j$ or $j_2\nsim j_1$. If $j_2\nsim \al_j$, then $\al_j$ is an isolated vertex in $C_n^2[F'^c]$; and if $j_2\nsim j_1$, then $j_2\nsim i_1$, which implies that $i_1$ is an isolated vertex in $C_n^2[F'^c]$. This means that $C_n^2[{F'}^c]$ is disconnected, and hence $F'\in M(\D{3}{C_n^2})$. 
  
    Since $i_1 < j_1 < j_2 <\alpha_j$, we observe that  $F'$ does not satisfy any of the conditions \ref{d1}-\ref{d4}, and therefore $F' \not\in \cD$.    We have $F'\in\T_{j_0}$. So, $i_1<j_1$ implies that  $F'\ll F_j$  by \Cref{definition:order_ll} \ref{s=t}. Hence $F' \prec F_j$  by \Cref{def:prec} \ref{def 1}, and thus, $F_r=F'$ satisfies \eqref{eqn:super} by \Cref{proposition:hash}. 

    \item  $j_1, j_2 > \al_j$.
    
    Since $C_n^2[{F_j}^c]$ is disconnected, we get $j_2 > \al_j+ 3$. Thus, $\al_j\geq 2$ implies that $j_2 \nsim \al_j$. Further, $F_i^c = \{\al_j\} \sqcup \{i_1, i_2\}$. So we have either $i_1 < \al_j$, or $i_1 > \al_j$. 

          \begin{enumerate}
         \item  $i_1 < \al_j$. 
    
         Let $F'^c:=\{\al_j\}\sqcup\{i_1,j_2\}$ and $F''^c:=\{\al_j\}\sqcup\{i_1,j_1\}$. We show that if $F'\in M(\D{3}{C_n^2})$, then $F_r=F'$ satisfies \eqref{eqn:super}, otherwise $F_r=F''$ satisfies \eqref{eqn:super}.  First, assume that  $F'\in M(\D{3}{C_n^2})$.  Using the facts that $i_1<\al_j$ and $j_2>\al_j+3$, we see that $F'\notin\cD$. Also, $i_1<j_1$ implies that $F'\ll F_j$ by \Cref{definition:order_ll} \ref{s=t}. Therefore, $F'\prec F_j$ by \Cref{def:prec} \ref{def 1}.  Hence  $F_r = F'$ satisfies \eqref{eqn:super} by \Cref{proposition:hash}. 

         We now assume that  $F'\notin M(\D{3}{C_n^2})$ \textit{i.e.}, $C_n^2[F'^c]$ is connected. Since $j_2\nsim\al_j$,  we get $\al_j\sim i_1$ and $i_1\sim j_2.$ Thus, $j_2>j_1>\al_j>i_1$ implies that $i_1\geq\al_j-2$. We know that $i_1<i_2$ and $C_n^2[F_i^c]$ is disconnected. So $i_2>\al_j$. Since $i_1, i_2, j_1$ and $j_2$ are distinct, it follows that  $i_2\leq n-3$ or $j_1\leq n-3$. 
       
        We have  $i_2,j_1>\al_j$ and $\al_j\os i_2,j_1$.  If $\al_j<m$, then using \Cref{remark:order in S} \ref{part:i}, $i_2, j_1\geq 2m-\al_j$; and if  $\al_j>m$, then $i_2,j_1>\al_j>2m-\al_j$. Hence we conclude that $i_2, j_1\geq 2m-\al_j$. This means that $n-3\geq2m-\al_j\geq n-\al_j$.   Hence $\al_j\geq 3$.  Now  $i_1\geq\al_j-2$ and $i_1 \sim j_2$ imply that $\al_j=3$, $i_1=1$ and $j_2=n-1.$ 
       If $2m=n+1,$ then $2m-\al_j=n-2>n-3,$ a contradiction. So $2m=n.$ 
    Then $2m-\al_j=n-3$, and this implies that $j_1\in\{n-3,n-2\}$. Hence $j_1\geq 6$ (as $n\geq 9$). Therefore, since $i_1=1$ and $\al_j=3$, it follows  that $j_1\nsim i_1$ and $j_1\nsim\al_j$. Clearly, $C_n^2[F''^c]$ is disconnected \textit{i.e.}, $F''\in M(\D{3}{C_n^2})$.
        Now, using the facts that $i_1=\al_j-2$ and  $j_1\geq \al_j+3$, we have $F'' \notin \cD$. 
      Also $i_1<j_1$ implies that $F''\ll F_j$  by \Cref{definition:order_ll} \ref{s=t}. Thus $F'' \prec F_j$  by \Cref{def:prec} \ref{def 1}. Hence $F_r=F''$ satisfies \eqref{eqn:super} by \Cref{proposition:hash}.
         
    \item $i_1 > \al_j$. 

          Let ${F'}^c :=\{\al_j\}\sqcup \{i_1,j_2\}$.  
          Suppose $C_n^2[{F'}^c]$  is connected. Then $ 2\leq \al_j< i_1 < j_1 < j_2$ implies that $i_1 \leq \al_j +2$, $j_1 = i_1+1$ and $j_2 = i_1+2$.     
         If $i_1 = \al_j + 1$, then $j_1 = \al_j + 2$ and $j_2 = \al_j + 3$, which implies that $C_n^2[{F_j}^c]$ is connected, a contradiction.
         So $i_1 = \al_j + 2$. This means that  $j_1 = \al_j + 3$ and $j_2 = \al_j + 4$. Since $i_1<i_2$ and $i_1,i_2,j_1,j_2$ are distinct, $\al_j\leq n-6$. Further,  by \Cref{remark:order in S} \ref{part:i} and \ref{part:ii}, $i_1>\al_j$ and $\al_j\os i_1$ implies that $i_1\geq 2m-\al_j$ \textit{i.e.}, $\al_j+2\geq2m-\al_j$. Hence, $\al_j\geq m-1$. This implies that $F_j$ satisfies \ref{d4}, which contradicts our assumption that $F_j\notin\cD$. Therefore $C_n^2[{F'}^c]$ is disconnected  and hence $F'\in M(\D{3}{C_n^2})$.

        Since $\al_j<i_1 < j_1 < j_2$, we have $F' \notin \cD$. Further, $i_1<j_1$ implies that $F'\ll F_j$  by \Cref{definition:order_ll} \ref{s=t}, and thus $F' \prec F_j$  by \Cref{def:prec} \ref{def 1}. Therefore $F_r=F'$ satisfies \eqref{eqn:super}.
    
     \end{enumerate}

 \item  $j_1 < \al_j$ and $j_2 > \al_j$.  
 
  Let ${F'}^c :=\{\al_j\}\sqcup\{i_1,j_2\}$, ${F''}^c :=\{\al_j\}\sqcup\{i_1,j_1\}$ and $F'''^c:=\{\al_j\}\sqcup\{j_1, i_2\}$. 
  
  First, let $F''\in M(\D{3}{C_n^2})$. Since $i_1<j_1$, we have $F''\ll F$  by \Cref{definition:order_ll} \ref{s=t}. Therefore, if $F''\notin\cD$, then  $F'' \prec F_j$ by \Cref{def:prec} \ref{def 1}. Hence, by taking $F_r=F''$, \eqref{eqn:super} is satisfied.

Suppose $F''\in\cD$. Since $i_1<j_1<\al_j$, it follows that $F''$ satisfies \ref{d3}. This means that $i_1=\al_j-4$, $j_1=\al_j-3$, and $\al_j=m+1$. Note that $n\geq 9$ implies $m\geq 5$. So $\al_j\geq 6$. Now,  $i_1=\al_j-4\geq 2$  and $j_2>\al_j$  implies that  $i_1\nsim\al_j$ and $i_1\nsim j_2$.
 Hence   $F'\in M(\D{3}{C_n^2})$  and $F'\notin\cD$. Also, since $i_1<j_1$ implies that $F'\ll F_j$, we get $F' \prec F_j$ by \Cref{def:prec} \ref{def 1}. Therefore,  taking $F_r=F'$, \eqref{eqn:super} is satisfied by \Cref{proposition:hash}.

Now, let  $F''\notin M(\D{3}{C_n^2})$, {\it i.e.,} $C_n^2[F''^c]$ is connected. We have $i_1<j_1<\al_j$ and $2\leq\al_j\leq n-2$. Suppose $j_1\nsim i_1$ or $j_1\nsim\al_j$. Then $i_1\sim\al_j$ such that $i_1=0$ and $\al_j=n-2$. Therefore, $\al_j\os i_2$ implies that  $i_2\in\{1,2,n-1\}$  by definition of $\cO$. We get a contradiction to the fact that $C_n^2[F_i^c]$ is disconnected.
 
Thus, $j_1\sim i_1$ and $j_1\sim\al_j$. Now, since $C_n^2[F_j^c]$ is disconnected,  $j_2\nsim\al_j$, which implies that $j_2\geq\al_j+3$.  
 We have ${F'}^c =\{\al_j\}\sqcup\{i_1,j_2\}$. Let  $F'\in M(\D{3}{C_n^2})$. Since $i_1<j_1<\al_j$ and $j_2 \geq \al_j+3$, $F'\notin\cD$. Also, $i_1<j_1$ implies that $F'\ll F_j$, and thus $F'\prec F_j$ by \Cref{def:prec} \ref{def 1}. Hence,  by taking $F_r=F'$, \eqref{eqn:super} is satisfied.
 
We now assume that $F'\notin M(\D{3}{C_n^2})$, \textit{i.e.}, $C_n^2[F'^c]$ is connected.  Since $j_2\nsim\al_j$, we get $j_2\sim i_1$ and $i_1\sim\al_j$. Then $i_1<j_1<\al_j<j_2$ implies that $i_1=\al_j-2$ and $j_1=\al_j-1$.  
Also, $j_2\sim i_1$ implies that  either $j_2=n-2$ and $i_1=0$, or $j_2=n-1$ and $i_1\in\{0,1\}$. If $i_1=0$, then $\al_j=2$, and thus, $\al_j\os i_2$ implies that  $i_2\in\{1,n-2,n-1\}$  by definition of $\cO$. This means that $C_n^2[F_i^c]$ is connected, a contradiction. 
So,  $i_1=1$, $j_1=2$, $\al_j=3$ and  $j_2=n-1$. Further, $C_n^2[F_i^c]$ is disconnected  implies that $i_2\nsim\al_j$, and since $i_2\neq j_2$,  $5<i_2<n-1=j_2$.
It follows that $i_2\nsim j_1$. We have ${F'''}^c =\{\al_j\}\sqcup\{j_1,i_2\}$.  Clearly, $C_n^2[F'''^c]\in M(\D{3}{C_n^2})$. Since $j_1=\al_j-1$, $i_2\geq\al_j+3$ and $\al_j=3\leq m-2$ (as $m\geq 5$), we have $F'''\notin\cD$.
Also, $i_2<j_2$ implies that $F'''\ll F_j$, and thus $F'''\prec F_j$ by \Cref{def:prec} \ref{def 1}. Hence,  $F_r=F'''$ satisfies \eqref{eqn:super}.

\end{enumerate}

\end{proof}

\begin{lemma}\label{case:B}
Suppose that $F_j\notin\cD$. If  $F_i\notin\T_{j_0}$ $(\textit{i.e.,}~ i_0\neq j_0)$, then there exists an $F_r\in M(\Delta_3(C_n^2))$ that satisfies \eqref{eqn:super} for the pair $F_i\prec F_j$.
\end{lemma}

\begin{proof}

We have ${F_i}^c = \{\al_i\}\sqcup\{i_1,i_2\}$ and ${F_j}^c = \{\al_j\}\sqcup\{j_1,j_2\}$  such that $F_i\prec F_j$. Then $F_j\notin\cD$ and $F_i\notin\T_{j_0}$ implies that $\al_i\os \al_j$ by \Cref{definition:order_ll} \ref{s<t} and \Cref{def:prec}. Therefore, since $\al_j\os j_1,j_2$ by \Cref{remark:Aj}, we get $\al_i\os j_1,j_2$.

Let ${F_q}^c=\{\al_i\}\sqcup\{j_1, j_2\}$. Observe that if $F_q\in M(\D{3}{C_n^2})$, then  $F_q\prec F_j$ by \Cref{proposition:precedes}, and hence $F_r=F_q$ satisfies \eqref{eqn:super} by \Cref{proposition:hash}. 
    
    Clearly,  $\al_i\os \al_j$ implies that  $\al_j\neq m$. Thus, we have two cases: (A)  $\al_j<m$, and (B) $\al_j>m$.\\
      
    \noindent{\bf Case A:}   $\al_j<m$.

 Since $\al_i\os \al_j$, we get $\al_j<\al_i<2m-\al_j$ by \Cref{remark:order in S} \ref{part:iii}.  We have $2m-\al_j>m$. It follows that either $\al_j<\al_i<m$, or $m\leq \al_i<2m-\al_j$. \\

    \noindent{\bf Subcase A.1:}  $\al_j<\al_i<m$.

   In this case, we show that $F_q\in M(\Delta_3(C_n^2))$, and thus $F_r=F_q$ satisfies \eqref{eqn:super}. 
   
      \begin{enumerate}[label=(\MakeUppercase{\roman*})]
    
        \item $j_1,j_2 < \al_j$.

       We have $j_1<j_2<\al_j<\al_i<m<n-1$. Thus $j_1\nsim\al_i$. Also, since $C_n^2[{F_j}^c]$ is disconnected,  $j_1\nsim j_2$ or $j_2\nsim\al_j$. If $j_1\nsim j_2$, then $j_1$ is an isolated vertex in $C_n^2[F_q^c]$; and if $j_2\nsim\al_j$, then $j_2\nsim\al_i$, which implies that $\al_i$ is an isolated vertex in $C_n^2[F_q^c]$. It follows that $F_q\in M(\Delta_3(C_n^2))$.

        \item $j_1,j_2 > \al_j$.

         Since $\al_j<\al_i<m$, $\al_j\leq m-2$. Further, $\al_j\os j_1$ and  $j_1>\al_j$ implies  that $j_1\geq  m+2$ by definition of $\cO$. We have $\al_j<\al_i<m<m+2\leq j_1<j_2$. Thus  $j_1\nsim\al_i$. Since $C_n^2[{F_j}^c]$ is disconnected,  $j_1\nsim j_2$ or $j_2\nsim\al_j$. If $j_1\nsim j_2$, then $j_1$ is an  isolated vertex in $C_n^2[F_q^c]$; and if $j_2\nsim\al_j$, then $j_2\nsim\al_i$, which implies that $\alpha_i$  is an  isolated vertex in $C_n^2[F_q^c]$. Hence $F_q\in M(\Delta_3(C_n^2))$.

        \item $j_1<\al_j$ and $j_2>\al_j$.       

        Since $\al_j<\al_i<m$,  $\al_j\os j_2$ and  $j_2>\al_j$, we get $j_2\geq  m+2$. We have $j_1<\al_j<\al_i<m<m+2\leq j_2$. Thus $j_2\nsim\al_i$. Since $C_n^2[{F_j}^c]$ is disconnected,  $j_1\nsim j_2$ or $j_1\nsim\al_j$.  Hence $F_q\in M(\Delta_3(C_n^2))$ (as $j_1\nsim\al_j$ implies that $j_1\nsim\al_i$).

    \end{enumerate}

    \noindent{\bf Subcase A.2:}   $m\leq \al_i<2m-\al_j$.

   In this case,  we show that either $F_q\in M(\D{3}{C_n^2})$ and thus $F_r=F_q$ satisfies \eqref{eqn:super}, or $F_q\notin M(\D{3}{C_n^2})$ and there exists some $F_r\neq F_q$ that satisfies \eqref{eqn:super} for the pair $F_i \prec F_j$. 

    \begin{enumerate}[label=(\MakeUppercase{\roman*})]
    
        \item $j_1, j_2 <\al_j$.
        
     Observe that since $j_1,j_2<\al_j$, if $\al_j<4$, then $C_n^2[F_j^c]$ is connected. So $\al_j \geq 4$, and hence $\al_i<2m-\al_j\leq 2m-4\leq n-3$. We have $j_1<j_2<\al_j<m\leq\al_i<n-3$. Thus $j_1\nsim\al_i$. Since $C_n^2[F_j^c]$ is disconnected, $j_1\nsim j_2$ or $j_2\nsim\al_j$.  If $j_1\nsim j_2$, then $j_1$ is an  isolated vertex in $C_n^2[F_q^c]$; and if $j_2\nsim\al_j$, then $j_2\nsim\al_i$, which implies that $\alpha_i$  is an  isolated vertex in $C_n^2[F_q^c]$. Hence $F_q\in M(\D{3}{C_n^2})$.

        \item  $j_1,j_2 > \al_j$.
        
          Using the facts that $j_1,j_2>\al_j$,  $\al_j<m$, and $\al_j\os j_1,j_2$, it follows that $j_1,j_2\geq 2m-\al_j>\al_i$ by \Cref{remark:order in S} \ref{part:i}. Thus $\al_j<m\leq\al_i<j_1<j_2$.
        
          Observe that if $j_1\nsim\al_i$ or $j_1\nsim j_2$, then  $\al_i\geq m$ implies that $j_2\nsim \al_i$, and hence $F_q\in M(\D{3}{C_n^2})$. So assume that $j_1\sim\al_i$ and $j_1\sim j_2$. Then $F_q\notin M(\D{3}{C_n^2})$. Let $F'^c:=\{\al_i\}\sqcup\{\al_j,j_2\}$.  We first show that $F'\in M(\D{3}{C_n^2})$. 
        
        Since $C_n^2[F_j^c]$ is disconnected, $j_1\sim j_2$ implies that  $j_1\nsim\al_j$ and $j_2\nsim\al_j$. Therefore, if $\al_i\nsim\al_j$ or $\al_i\nsim j_2$, then $F'\in M(\D{3}{C_n^2})$. Let $\al_i\sim\al_j$.
        We need to show that $\al_i\nsim j_2$. 
        Since $\al_j<m$ and $\al_i\geq m$, $\al_i\sim\al_j$ implies that $\al_j\in\{m-2, m-1\}$. If $\al_j=m-2$, then  $\al_i=m$ and $j_2>j_1\geq m+2$ (as $\al_j\os j_1$ and $j_1>\al_j$). Thus $\al_i\nsim j_2$.
       
        If $\al_j=m-1$, then  $\al_i<_{\cO}\al_j$ implies that $\al_i=m$. 
        Further, we have $j_1\sim\al_i$, $j_1\nsim\al_j$, and $j_1 > \al_i=m$. Thus $j_1=m+2$,  and hence $j_2>m+2$. This means that $\al_i\nsim j_2$.

       Hence  $F'\in M(\D{3}{C_n^2})$.   
   Since $\al_i\os\al_j$, from \Cref{proposition:precedes}, we see that $F'\prec F_j$. Thus $F_r=F'$ satisfies \eqref{eqn:super} by \Cref{proposition:hash}.

        \item $j_1 < \al_j$  and $j_2 > \al_j$.
       
         We have $j_2>\al_j$, $\al_j<m$, and $\al_j\os j_2$. Thus, we get $j_2\geq 2m-\al_j>\al_i$ by \Cref{remark:order in S} \ref{part:i}. So,  $j_1<\al_j<m\leq\al_i<j_2$.
        
        Suppose $j_2\nsim\al_i$. Now, since $C_n^2[{F_j}^c]$ is disconnected, we have $j_1\nsim j_2$ or $j_1\nsim\al_j$. If $j_1\nsim j_2$, then $j_2$ is an  isolated vertex in $C_n^2[F_q^c]$; and if $j_1\nsim\al_j$, then $j_1\nsim\al_i$, which implies that $\alpha_i$  is an  isolated vertex in $C_n^2[F_q^c]$. Thus, we see that  $F_q\in M(\D{3}{C_n^2})$. 
       
        Now, we assume that $j_2\sim\al_i$. 
       Suppose $j_1\sim\al_i$. Then, $j_1<\al_j<m\leq\al_i<j_2$ implies that either $j_1=m-2$, $\al_j=m-1$, $\al_i=m$ and $j_2\in\{m+1,m+2\}$, or $j_1=0$, $\al_i=n-2$, and $j_2=n-1$. In the former case, since $C_n^2[{F_j}^c]$ is disconnected, $j_2=m+2=\al_j+3$. This means that $F_j$ satisfies \ref{d2}, a contradiction as $F_j\notin\cD$. Hence, we have $j_1=0$, $\al_i=n-2$,  and $j_2=n-1$. Now, $\al_i\os\al_j$ and $\al_j<m$ implies that $\al_j\leq 2$. It follows that $C_n^2[{F_j}^c]$ is connected, again a contradiction. Therefore, $j_1\nsim\al_i$. Now, if $j_1\nsim j_2$, then $F_q\in M(\D{3}{C_n^2})$. 
        
        So, let $j_1\sim j_2$. Then $F_q\notin M(\D{3}{C_n^2})$. Let $F'^c:=\{\al_i\}\sqcup\{j_1,\al_j\}$.  Since $C_n^2[F_j^c]$ is disconnected and $j_1\sim j_2$,  we get $j_1\nsim\al_j$. Therefore, $j_1\nsim\al_i$ implies that $F'\in M(\Delta_3(C_n^2))$. We have $\al_i\os\al_j$. Thus $F'\prec F_j$ by \Cref{proposition:precedes}. Hence $F_r=F'$ satisfies \eqref{eqn:super} by \Cref{proposition:hash}.        

    \end{enumerate}

  \noindent{\bf Case B:} 
     $\al_j>m$.

     We have ${F_q}^c=\{\al_i\}\sqcup\{j_1, j_2\}$. Recall that, if $F_q\in M(\D{3}{C_n^2})$, then $F_r=F_q$ satisfies \eqref{eqn:super}.
      Since $\al_i\os \al_j$, $2m-\al_j\leq\al_i<\al_j$ by \Cref{remark:order in S} \ref{part:iv}.  Further, $2m-\al_j<m$. This means  that either $2m-\al_j\leq \al_i\leq m$, or $m< \al_i<\al_j$. \\

    \noindent{\bf Subcase B.1:} 
    $2m-\al_j\leq \al_i\leq m$.

     In this case also, we show that either $F_q\in M(\D{3}{C_n^2})$, or $F_q\notin M(\D{3}{C_n^2})$ and we find some $F_r\neq F_q$ that satisfies \eqref{eqn:super}.

    \begin{enumerate}[label=(\MakeUppercase{\roman*})]
    
        \item $j_1, j_2 <\al_j$. 
         
        Since $\al_j\os j_1,j_2$ and $\al_j>m$,  it follows from \Cref{remark:order in S} \ref{part:ii} that $j_1, j_2 < 2m-\al_j\leq \al_i$. So, we have $j_1<j_2<\al_i\leq m<\al_j$.
        
          Observe that if $j_1\nsim j_2$ or $j_2\nsim\al_i$, then  $\al_i\leq m$ implies that $j_1\nsim \al_i$, and hence $F_q\in M(\D{3}{C_n^2})$. So assume that $j_1\sim j_2$ and $j_2\sim \al_i$. Then $F_q\notin M(\D{3}{C_n^2})$.  Let ${F'}^c:=\{\al_i\}\sqcup\{j_1, \al_j\}$. We first show that $F'\in M(\D{3}{C_n^2})$.
       
     Since $C_n^2[{F_j}^c]$ is disconnected, $j_1 \sim j_2$ implies that $j_1 \not\sim \alpha_j$. Therefore, if $\al_i\nsim\al_j$ or $\al_i\nsim j_1$, then $F'\in M(\D{3}{C_n^2})$. 
       Let $\al_i\sim\al_j$. We show that $\al_i\nsim j_1$. 
    Suppose $\al_i\sim j_1$.
         Since $\al_i\leq m$ and $\al_j>m$, $\al_i\sim\al_j$ implies that  $\al_j\in\{m+1, m+2\}$. 
          If $\al_j=m+1$, then  $\al_i\in\{m-1,m\}$. Using the facts that $\al_j\os j_2$ and $j_2<\al_j$, it follows that $j_2\leq m-2$. Hence $j_1<m-2$. Then $\al_i=m-1$, $j_1=m-3=\al_j-4$ and $j_2=m-2=\al_j-3$. This means that $F_j$ satisfies \ref{d3}, a contradiction.

        If $\al_j=m+2$, then  $\al_i=m$ and $j_1<j_2\leq m-3$ (as $\al_j\os j_2$ and $j_2<\al_j$). Thus $\al_i\nsim j_1$,  which contradicts our assumption that $\al_i\sim j_1$. Since we get a contradiction in each case, our assumption that $\al_i\sim j_1$ is false. Therefore, $\al_i\nsim j_1$. Hence  $F'\in M(\D{3}{C_n^2})$. 
    
        Since $\al_i\os\al_j$, $F\prec F_j$ by \Cref{proposition:precedes}. Thus $F_r=F'$ satisfies \eqref{eqn:super}.  

        \item  $j_1,j_2 > \al_j$.

If $\alpha_j \geq n-4$, then $j_1, j_2 > \alpha_j$ implies $C_n^2[F_j^c]$ is connected. So $\al_j \leq n-5$, and hence $\al_i\geq 2m-\al_j\geq 5$. We have $j_2>j_1>\al_j>m\geq\al_i\geq 5$. Thus $j_2\nsim\al_i$. Since $C_n^2[{F_j}^c]$ is disconnected, $j_1\nsim j_2$ or $j_1\nsim\al_j$. Hence   $F_q\in M(\D{3}{C_n^2})$  (as $j_1\nsim\al_j$ implies that $j_1\nsim\al_i$).

      \item $j_1 < \al_j$  and $j_2 > \al_j$.

         Since $\al_j\os j_1$ and $\al_j>m$, we get $j_1<2m-\al_j\leq \al_i$ by \Cref{remark:order in S} \ref{part:ii}. So, we have $j_1<\al_i\leq m<\al_j<j_2$.

          Since $C_n^2[{F_j}^c]$ is disconnected, $j_1\nsim j_2$ or $j_2\nsim\al_j$. Now, if $j_1\nsim\al_i$, then  $F_q\in M(\D{3}{C_n^2})$ (as $j_1\nsim\al_i$ and $j_2\nsim\al_j$ imply that $j_2\nsim\al_i$).

        We now assume that $j_1\sim\al_i$. 
       Suppose $j_2\sim\al_i$.  Then $j_1<\al_i\leq m<\al_j<j_2$ implies that either $\al_i=m$, $\al_j=m+1$, $j_2=m+2$  and $j_1\in\{m-2,m-1\}$, or $\al_i=1$, $j_2=n-1$, and $j_1=0$. In the former case, since $C_n^2[{F_j}^c]$ is disconnected, $j_1=m-2=\al_j-3$. This means that $F_j$ satisfies \ref{d1}, a contradiction.
        So, we have $\al_i=1$, $j_2=n-1$, and $j_1=0$. Then $\al_i\os\al_j$  and $\al_j>m$ implies that $\al_j=n-1=j_2$, which is again a contradiction. Therefore, $j_2\nsim\al_i$. Now, if $j_1\nsim j_2$, then $F_q\in M(\D{3}{C_n^2})$.

        So, let $j_1\sim j_2$. Then $F_q\notin M(\D{3}{C_n^2})$.
         Let $F'^c:=\{\al_i\}\sqcup\{\al_j,j_2\}$. Since $C_n^2[F_j^c]$ is disconnected and $j_1\sim j_2$, we get $j_2\nsim\al_j$. Therefore, $j_2\nsim\al_i$ implies that $F'\in M(\Delta_3(C_n^2))$.   Further, since $\al_i\os\al_j$, we get $F'\prec F_j$ by \Cref{proposition:precedes}. Hence $F_r=F'$ satisfies \eqref{eqn:super} by \Cref{proposition:hash}.
    
    \end{enumerate}

    \noindent{ \bf Subcase B.2:} 
    $m< \al_i<\al_j$.
    
    We show that $F_q\in M(\D{3}{C_n^2})$, and thus $F_r=F_q$ satisfies \eqref{eqn:super}.

    \begin{enumerate}[label=(\MakeUppercase{\roman*})]
    
         \item $j_1,j_2 < \al_j$.
         
        Since $m<\al_i<\al_j$, we have $\al_j\geq m+2$. Then $\al_j\os j_2<\al_j$ implies that $j_2\leq  m-3$. It follows that $j_1<j_2\leq m-3<m<\al_i<\al_j$. Thus $j_2\nsim\al_i$. Further, since $C_n^2[{F_j}^c]$ is disconnected, $j_1\nsim j_2$ or $j_1\nsim\al_j$. This implies that $F_q\in M(\Delta_3(C_n^2))$ (as $j_1\nsim\al_j$ implies that $j_1\nsim\al_i$).

        \item $j_1,j_2 > \al_j$.
        
        We have $j_2>j_1>\al_j>\al_i>m>1$. So $j_2\nsim\al_i$. Since $C_n^2[{F_j}^c]$ is disconnected, $j_1\nsim j_2$ or $j_1\nsim\al_j$. Hence $F_q\in M(\Delta_3(C_n^2))$ (as $j_1\nsim\al_j$ implies that $j_1\nsim\al_i$).

        \item $j_1<\al_j$ and $j_2>\al_j$.
        
        Since $m< \al_i<\al_j$ and  $\al_j\os j_1<\al_j$, we get $j_1\leq  m-3$. Hence $j_1\leq m-3<m<\al_i<\al_j<j_2$. Thus $j_1\nsim\al_i$. Now, since $C_n^2[{F_j}^c]$ is disconnected, we have $j_1\nsim j_2$ or $j_2\nsim\al_j$. This implies that $F_q\in M(\Delta_3(C_n^2))$ (as $j_2\nsim\al_j$ implies that $j_2\nsim\al_i$).
        
    \end{enumerate}

    \end{proof}

In the following Lemma,  we now consider the case when $F_j\in\cD$.

\begin{lemma}\label{case:D}
   Suppose that $F_j\in\cD$. Then there exists an $F_r\in M(\Delta_3(C_n^2))$ that satisfies \eqref{eqn:super} for the pair $F_i\prec F_j$.
\end{lemma}

\begin{proof}
    
We have ${F_i}^c = \{\al_i\}\sqcup\{i_1,i_2\}$,  ${F_j}^c = \{\al_j\}\sqcup\{j_1,j_2\}$  and $F_i\prec F_j$.  If $F_i\in\T_{j_0}$, {\it i.e.,} $i_0 = j_0$, then we have $\al_i=\al_j$. In this case, if $\{i_1, i_2\}\cap\{j_1, j_2\}\neq\emptyset$, then $F_r=F_i$ satisfies \eqref{eqn:super}. So we assume that $\{i_1, i_2\}\cap\{j_1, j_2\}=\emptyset$ whenever $F_i\in\T_{j_0}$.
    Moreover, if  $F_i\notin\T_{j_0}$, then $F_i\prec F_j$  and $F_j\in\cD$ imply that $\al_i\os \al_j$ by  by \Cref{definition:order_ll} \ref{s<t} and \Cref{def:prec}.
    Note that since $n\geq 9$, we have $5\leq m\leq n-4$.
       Now, $F_j\in\cD$ implies that $F_j$ satisfies one of the conditions from \ref{d1} to \ref{d4}. 

    \begin{enumerate}[label=(\arabic*)]

    \item $F_j$ satisfies \ref{d1}. 
    
   We have ${F_j}^c =\{\al_j\}\sqcup\{\al_j-3,\al_j+1\}$ and $\al_j=m+1$. Hence, $6\leq\al_j\leq n-3$.
    
     Let   Let $F'^c:=(F_j^c\setminus\{j_2\})\sqcup\{i_1\}$ and  $F''^c:=(F_j^c\setminus\{j_2\})\sqcup\{i_2\}$.
    First, we show that if $i_1<j_1$, then $F_r=F'$ satisfies \eqref{eqn:super} and if 
$i_2>j_2$, then $F_r=F''$ satisfies \eqref{eqn:super}. Later, we discuss the case when  $i_1\geq j_1$ and $i_2\leq j_2$.

First, suppose $i_1<j_1$.  Since $\al_j\os j_1$, $j_1<\al_j$ and $\al_j>m$, we get $j_1<2m-\al_j$ by \Cref{remark:order in S} \ref{part:ii}, and hence $i_1<2m-\al_j$, thereby implying that $\al_j\os i_1$.  This means that $F'^c=\{\al_j\}\sqcup\{i_1,j_1\}$.  Since $i_1<j_1=\al_j-3<\al_j\leq n-3$, we get $\al_j\nsim i_1$ and $\al_j\nsim j_1$. Hence $F'\in M(\D{3}{C_n^2})$. Also, $i_1<j_1$ implies that $F'\ll F_j$.  Therefore, $F'\prec F_j$ by \Cref{def:prec} \ref{def 3} and \ref{def 4}. Hence $F_r=F'$ satisfies \eqref{eqn:super} by \Cref{proposition:hash}.

Let  $i_2>j_2$.  We have $i_2>j_2>\al_j>m$. Therefore $\al_j\os i_2$, and thus,  $\al_j\os j_1$ and $j_1<\al_j<i_2$ implies that $F''^c=\{\al_j\}\sqcup\{j_1,i_2\}$. Since $\al_j\geq 6$, $j_1 = \alpha_j-3 \geq 3$. So we have $3\leq j_1=\al_j-3<\al_j<\al_j+1=j_2<i_2$. It follows that $j_1\nsim \al_j$ and  $j_1\nsim i_2$. Hence $F''\in M(\D{3}{C_n^2})$.
Since $j_1=\al_j-3$ and $i_2\geq\al_j+2$, $F''\notin\cD$. Therefore $F''\prec F_j$ by \Cref{def:prec} \ref{def 3}, and thus  $F_r=F''$ satisfies \eqref{eqn:super} by \Cref{proposition:hash}.

Now, suppose $i_1\geq j_1$ and $i_2\leq j_2$. If $F_i\in\T_{j_0}$, then $\al_j=\al_i\neq i_2$ and $\{i_1,i_2\}\cap\{j_1,j_2\}=\emptyset$.  Thus, $i_2\leq j_2=\al_j+1$ implies that $i_2<\al_j$. Therefore, we have $\al_j-3=j_1<i_1<i_2<\al_j$. It follows that $C_n^2[F_i^c]$ is connected, a contradiction.  Hence $F_i\notin \T_{j_0}$.

    We know that $\al_i<_{\cO} \al_j=m+1$. Thus, $\al_i\in\{m-1,m\}$.
    We have $m-2=j_1\leq i_1<i_2\leq j_2=m+2$. Therefore, since $C_n^2[F_i^c]$ is disconnected, we get $\al_i=m-1$, $i_1=j_1=m-2$ and $i_2=j_2=m+2$. Hence $F_r=F_i$ satisfies \eqref{eqn:super}.
    
    \item $F_j$ satisfies \ref{d2}.
    
     We have ${F_j}^c =\{\al_j\}\sqcup\{\al_j-1,\al_j+3\}$, where $\al_j=m-1$.  Therefore $5\leq m\leq n-4$ implies that $4\leq\al_j\leq n-5$.

    Suppose $j_1\leq i_1<i_2\leq j_2$.
    If $F_i\in\T_{j_0}$, then we have $\al_j=\al_i\neq i_1$ and $\{i_1,i_2\}\cap\{j_1,j_2\}=\emptyset$. This implies that $\al_j<i_1<i_2<j_2=\al_j+3$. It follows that $C_n^2[F_i^c]$ is connected, a contradiction.  So $F_i\notin \T_{j_0}$.

   We have $m-2=j_1\leq i_1<i_2\leq j_2=m+2$.  Moreover, since $\al_i\os \al_j=m-1$, we get  $\al_i=m$. It follows that $C_n^2[F_i^c]$ is connected, again a contradiction.
   Hence our assumption that $j_1\leq i_1< i_2\leq j_2$ is false.

   Now,  we consider two cases: $i_1<j_1$ and $i_1\geq j_1$.

Suppose $i_1<j_1$.  Let  $F'^c:=(F_j^c\setminus\{j_1\})\sqcup\{i_1\}$. We have $i_1<j_1<\al_j<m$. Therefore $\al_j\os i_1$, and thus,  $\al_j\os j_2$ and $i_1<\al_j<j_2$ implies that $F'^c=\{\al_j\}\sqcup\{i_1,j_2\}$. First, we show that $F'\in M(\D{3}{C_n^2})$. We have $i_1<j_1=\al_j-1<\al_j<\al_j+3=j_2$. So $j_2\nsim\al_j$. Now, if $i_1\nsim \al_j$, then $F'\in M(\D{3}{C_n^2})$. So, let $i_1\sim \al_j$. Then  $i_1=\al_j-2\geq 2$ (as $\al_j\geq 4$). Thus $i_1\nsim j_2$. Hence $F'\in M(\D{3}{C_n^2})$.
    Further, since $i_1<j_1=\al_j-1$, $F'\notin\cD$. Thus $F'\prec F_j$ by \Cref{def:prec} \ref{def 3}. Therefore $F_r=F'$ satisfies \eqref{eqn:super} by \Cref{proposition:hash}.

 Now, suppose that $i_1\geq j_1$. If $i_2\leq j_2$, then $j_1\leq i_1< i_2\leq j_2$, which is a contradiction. So $i_2>j_2$. Let $F''^c:=(F_j^c\setminus\{j_2\})\sqcup\{i_2\}$. Since $\al_j\os j_2$, $j_2>\al_j$ and $\al_j<m$, we get $j_2\geq 2m-\al_j$ by \Cref{remark:order in S} \ref{part:i}, and hence $i_2>2m-\al_j$, thereby implying that $\al_j\os i_2$.  This means that $F''^c=\{\al_j\}\sqcup\{j_1,i_2\}$. Since $\al_j\geq 4$, we have $j_1\geq 3$. Then $3\leq j_1=\al_j-1<\al_j<\al_j+3=j_2<i_2$. It follows that $i_2\nsim j_1$  and $i_2\nsim\al_j$. Hence, $F''\in M(\D{3}{C_n^2})$.
Since $j_1=\al_j-1$ and $i_2\geq\al_j+4$, $F''\notin\cD$. Thus $F''\prec F_j$ by \Cref{def:prec} \ref{def 3} and $F_r=F''$ satisfies \eqref{eqn:super}. 

    \item $F_j$ satisfies \ref{d3}.
     
   We have ${F_j}^c =\{\al_j\}\sqcup\{\al_j-4,\al_j-3\}$, where $\al_j=m+1$.  Therefore $6\leq\al_j\leq n-3$ (as $5\leq m\leq n-4$).
    
     Let $F'^c:=(F_j^c\setminus\{j_2\})\sqcup\{i_1\}$ and  $F''^c:=(F_j^c\setminus\{j_2\})\sqcup\{i_2\}$. 
    First we show that if $i_1<j_1$, then $F_r=F'$ satisfies \eqref{eqn:super} and if $i_2>\al_j$, then $F_r=F''$ satisfies \eqref{eqn:super}. Later, we discuss the case when $i_1\geq j_1$ and $i_2\leq \al_j$.

    Suppose  $i_1<j_1$.  Since $\al_j\os j_1$, $j_1<\al_j$ and $\al_j>m$, we get $j_1<2m-\al_j$ by \Cref{remark:order in S} \ref{part:ii}, and hence $i_1<2m-\al_j$, thereby implying that $\al_j\os i_1$.  This means that $F'^c=\{\al_j\}\sqcup\{i_1,j_1\}$.  Since $i_1<j_1=\al_j-4<\al_j\leq n-3$, we get $\al_j\nsim i_1$  and $\al_j\nsim j_1$. Hence $F'\in M(\D{3}{C_n^2})$.  Further, since $i_1<j_1=\al_j-4$, $F'\notin\cD$. Thus $F'\prec F_j$ by \Cref{def:prec} \ref{def 3}.  Hence $F_r=F'$ satisfies \eqref{eqn:super} by \Cref{proposition:hash}.

 Let $i_2>\al_j$.  Since $\al_j>m$, we get $\al_j\os i_2$. Thus,  $\al_j\os j_1$ and $j_1<\al_j<i_2$ implies that $F''^c=\{\al_j\}\sqcup\{j_1,i_2\}$. Since $\al_j\geq 6$,  we have $2\leq j_1=\al_j-4<\al_j<i_2$. So, $j_1\nsim \al_j$ and $j_1\nsim i_2$. Thus $F''\in M(\D{3}{C_n^2})$.
Since $j_1=\al_j-4$ and $i_2\geq\al_j+1$, it follows that $F''\notin\cD$. Hence $F''\prec F_j$ by \Cref{def:prec} \ref{def 3}. Therefore $F_r=F''$ satisfies \eqref{eqn:super}.

Now, let $i_1\geq j_1$ and $i_2\leq \al_j$. Suppose $F_i\in\T_{j_0}$. Then $\al_j=\al_i\neq i_2$ and $\{i_1,i_2\}\cap\{j_1,j_2\}=\emptyset$. Thus, $i_1\geq j_1=j_2-1$ implies that $i_1>j_2$. So, we have $\al_j-3=j_2<i_1<i_2<\al_j$. It follows that $C_n^2[F_i^c]$ is connected, a contradiction. Hence $F_i\notin \T_{j_0}$.
   
   Since $\al_i<_{\cO} \al_j=m+1$, we get $\al_i\in\{m-1,m\}$.
    We have $m-3=j_1\leq i_1<i_2\leq \al_j=m+1$. Therefore,  $C_n^2[F_i^c]$ is disconnected implies that $\al_i=m$, $i_1=j_1=m-3$ and $i_2=\al_j=m+1$. Hence $F_r=F_i$ satisfies \eqref{eqn:super}.
    
    \item $F_j$ satisfies \ref{d4}. 
    
    We have ${F_j}^c =\{\al_j\}\sqcup\{\al_j+3,\al_j+4\}$, where $\al_j\in\{m-1, m, \dots, n-6\}$.  Since $m\geq 5$, we have $4\leq\al_j\leq n-6$.

   Let  $F'^c:=(F_j^c\setminus\{j_1\})\sqcup\{i_1\}$ and $F''^c:=(F_j^c\setminus\{j_1\})\sqcup\{i_2\}$.
    First we show that if $i_1<\al_j$, then $F_r=F'$ satisfies \eqref{eqn:super} and if $i_2>j_2$, then $F_r=F''$ satisfies \eqref{eqn:super}. Later, we discuss the case when $i_1\geq \al_j$ and $i_2\leq j_2$.

    Suppose $i_1<\al_j$. We first show that $F'\in M(\D{3}{C_n^2})$. Since $\al_j\leq n-6$, we get $j_2 = \alpha_j +4 \leq n-2$. We have $i_1<\al_j<\al_j+4=j_2\leq n-2$.  So $j_2\nsim\al_j$. 
    Now, if $i_1\nsim \al_j$, then $F'\in M(\D{3}{C_n^2})$. So, let $i_1\sim \al_j$. Then  $i_1 \in \{\alpha_j-2, \al_j-1\}$ and therefore $i_1 \geq 2$. Thus $i_1\nsim j_2$. Hence $F'\in M(\D{3}{C_n^2})$.
   We now show that $F'\notin\cD$. If $\al_j\os i_1$, then $i_1<\al_j<j_2$ and $\al_j\os j_2$ implies that $F'^c=\{\al_j\}\sqcup\{i_1,j_2\}$.
        Therefore, since $i_1\leq\al_j-1$ and $j_2=\al_j+4$, it follows that $F'\notin\cD$.
        Moreover, if $i_1\os\al_j$, then $\al_j\os j_2$ implies that $i_1\os j_2$. Thus $F'^c=\{i_1\}\sqcup\{\al_j,j_2\}$. Since $j_2=\al_j+4>i_1+4$, we get $F'\notin\cD$. Thus $F'\prec F_j$ by \Cref{def:prec} \ref{def 3}. Therefore $F_r=F'$ satisfies \eqref{eqn:super}.

Now, suppose $i_2>j_2$.  If $\al_j<m$, then $\al_j\os j_2$ and $j_2>\al_j$ implies that $j_2\geq 2m-\al_j$ by \Cref{remark:order in S} \ref{part:i}. Hence $i_2>2m-\al_j$, thereby implying that $\al_j\os i_2$. Also, if $\al_j\geq m$, then $i_2>j_2>\al_j$ implies that $\al_j\os i_2$.  This means that $F'^c=\{\al_j\}\sqcup\{j_2,i_2\}$.  We have $4\leq\al_j<\al_j+4=j_2<i_2$. Thus $\al_j\nsim j_2$ and $\al_j\nsim i_2$. Hence, $F''\in M(\D{3}{C_n^2})$. Since $i_2>j_2=\al_j+4$, $F''\notin\cD$. Thus $F'\prec F_j$ by \Cref{def:prec} \ref{def 3}. Therefore $F_r=F'$ satisfies \eqref{eqn:super}.

Finally, let $i_1\geq \al_j$ and $i_2\leq j_2$. 
 Suppose $F_i\in\T_{j_0}$. Then $\al_j=\al_i\neq i_1$ and $\{i_1,i_2\}\cap\{j_1,j_2\}=\emptyset$.  Thus, $i_2\leq j_2=j_1+1$ implies that $i_2<j_1$. Therefore, we have $\al_j<i_1<i_2<j_1=\al_j+3$. This implies that $C_n^2[F_i^c]$ is connected,  a contradiction.  Hence $F_i\notin \T_{j_0}$.
 
  We have $\al_j\in\{m-1,m,\ldots,n-6\}$. 
  Since $F_i\notin \T_{j_0}$ implies that  $\al_i\os \al_j$, we have $\al_j\neq m$. We consider two cases: (i) $\al_j=m-1$, and (ii) $\al_j\in\{m+1, m+2, m+3, \ldots, n-6\}$.

 \begin{enumerate}[label=(\roman*)]
     \item $\al_j=m-1$. 
     
     Since $\al_i\os\al_j$, $\al_i=m$. We have $m-1=\al_j\leq i_1<i_2\leq j_2=\al_j+4=m+3$.
     Therefore, $C_n^2[F_i^c]$ is disconnected implies that $i_1=\al_j=m-1$ and $i_2=j_2=m+3$. Hence $F_r=F_i$ satisfies \eqref{eqn:super}.

     \item $\al_j\in\{m+1, m+2, m+3, \ldots, n-6\}$.

     Let $F'''^c:=(F_j^c\setminus\{\al_j\})\sqcup\{\al_i\}$. Since $\al_i\os\al_j$, we get $\al_i\os j_1,j_2$. Thus $F'''^c=\{\al_i\}\sqcup\{j_1,j_2\}$.
      We have $\al_i\os \al_j$ and $\al_j>m$. So by \Cref{remark:order in S} \ref{part:iv}, $\al_j>\al_i\geq 2m-\al_j\geq 2m-(n-6)\geq 6$. Then, we have $6\leq\al_i<\al_j<\al_j+3=j_1<j_2$.  It follows that $\al_i\nsim j_1$ and $\al_i\nsim j_2$. Hence $F'''\in M(\D{3}{C_n^2})$. Since $\al_i\os\al_j$, \Cref{proposition:precedes} implies that $F'''\prec F_j$. Hence  $F_r=F'''$ satisfies \eqref{eqn:super} by \Cref{proposition:hash}. 
 \end{enumerate}

\end{enumerate}
     
\end{proof}

 From Lemmas \ref{case:A}, \ref{case:B} and \ref{case:D}, it follows that for any $F_i, F_j\in M(\D{3}{C_n^2})$ with $F_i\prec F_j$, there exists an $F_r \prec F_j$ such that $F_i \cap F_j \subseteq F_r \cap F_j$ and $|F_r \cap F_j| = $ $|F_j| - 1$. This means that $\prec$ provides a shelling order for $\D{3}{C_n^2}$. 

\subsection{Spanning Facets}\label{subsection:spanning facets}

In this section, we characterize and count the spanning facets for the shelling order $\prec$. 

\begin{remark}\label{remark:F is a spanning facet}    Observe that by definition, any facet $F\in M(\Delta_3(C_n^2))$ is a spanning facet if and only if for each  $\mu \in F$, there exists $ \tilde{F} \in M(\Delta_3(C_n^2))$ such that $\tilde{F} \prec F$ and $\tilde{F}^c=(F^c\setminus\{\nu\})\sqcup\{\mu\}$ for some $\nu\in F^c$.
\end{remark}

Define a subset $\cS$ of $M(\Delta_3(C_n^2))$ as follows:  A facet $F\in\cS$ if and only if there exists $s$ such that $F\in\T_s$ and $F$ satisfies one of the following conditions: 

\begin{enumerate}[label=$(\mathcal{S}_{\arabic*})$]
    \item \label{s1} $F^c=\{\al_s\}\sqcup\{s_1,n-1\}$, where $\al_s=3$ and $s_1\in V(C_n^2)\setminus(\{0, 1, 2, 3, \ldots, 2m-4\}\sqcup\{n-1\})$.

      \item \label{s2}  $F^c=\{\al_s\}\sqcup\{s_1,n-1\}$, where $\al_s\in\{4,5,\ldots,m-2\}$ and $s_1\in V(C_n^2)\setminus(\{\al_s-4, \al_s-3, \al_s-2\}\sqcup\{\al_s, \al_s+1, \ldots, 2m-\al_s-1\}\sqcup\{n-1\})$.

     \item \label{s3} $F^c=\{\al_s\}\sqcup\{s_1,n-1\}$, where $\al_s\in\{m-1, m, m+1,m+2, \ldots, n-3\}$,  $s_1\in V(C_n^2)\setminus(\{\omega, \omega+1, \ldots,\al_s,\al_s+1,\al_s+2,\al_s+3 \text{ (mod  $n$})\}\sqcup\{y\})$,  $\omega = \min\{2m-\al_s, \al_s-4\}$ and $$y=\begin{cases}
    n-1 &\text{ if } \al_s<n-4,\\
    0 &\text{ if } \al_s=n-4,\\
    1 &\text{ if } \al_s>n-4.  
    \end{cases}$$ 

\end{enumerate}

Our aim is to show that a facet $F\in M(\Delta_3(C_n^2))$ is a spanning facet if and only if $F \in \cS$. 
\begin{proposition}\label{proposition:order ll}
    Let $F, F' \in M(\Delta_3(C_n^2))$. Let  $F \in \T_s$  such that $F^c=\{\al_s\}\sqcup\{s_1, n-1\}$ and ${F'}^c = \{\al_s, s_1, \mu\}$ for some   $\mu \in F$. If $\al_s\os\mu$, then  $F'\in\T_s$ and  $F'\ll F$. 
\end{proposition}

\begin{proof}
    Since $\al_s\os \mu,s_1$, we have $F'\in\T_s$. If $\mu<s_1$, then $F'^c=\{\al_s\}\sqcup\{\mu,s_1\}$ and $F'\ll F$  by \Cref{definition:order_ll} \ref{s=t}. On the other hand, if $\mu>s_1$, then $F'^c=\{\al_s\}\sqcup\{s_1,\mu\}$, and thus $\mu<n-1$ implies that $F'\ll F$ \Cref{definition:order_ll} \ref{s=t}.
\end{proof}

\begin{proposition}\label{spannig:s1s2}

      Let $F\in M(\D{3}{C_n^2})$ such that $F\in\T_s$  for some $s$ and $F^c=\{\al_s\}\sqcup\{s_1, n-1\}$. Suppose $\al_s\in \{3, 4, \ldots, m-2\}$ and $s_1\in \mathcal{U}=V(C_n^2)\setminus\bigl(\{\al_s-4 \ (\text{mod $n$}), \al_s-3, \al_s-2, \ldots, 2m-\al_s-1\}\cup\{n-1\}\bigr)$. Then $F$ is a spanning facet.
\end{proposition}

\begin{proof}
From \Cref{remark:F is a spanning facet}, it is sufficient to show that for each $\mu \in F$, there exists a facet  $\tilde{F}\in M(\D{3}{C_n^2}) $ such that $\tilde{F} \prec F$ and $\tilde{F}^c = (F^c \setminus \{\nu\}) \sqcup \{\mu\}$ for some $\nu \in F^c$. 

    Let $\mu\in F$ and $F'^c:=(F^c\setminus\{n-1\})\sqcup\{\mu\}=\{\al_s, s_1, \mu\}$.
    Since $\al_s\leq m-2$, we have $2m-\al_s-1\geq m+1\geq \al_s+3$, and hence $2m-\al_s-1\geq\al_s+3$. First, suppose that $s_1\neq\al_s+4$. Then $s_1\in \mathcal{U}$ implies that $s_1\notin\{\al_s-4 \text{ (mod $n$)}, \al_s-3, \al_s-2, \ldots, \al_s+4\}\cup\{n-1\}$. This means that $\al_s\nsim s_1$ and $|\al_s-s_1|\geq5$.
    Further, since $\al_s\geq3$ and $s_1<n-1$, we get $\mu\nsim\al_s$ or $\mu\nsim s_1$. Thus $F'\in M(\Delta_3(C_n^2))$. 
    If $\mu\os\al_s$, then $F'\prec F$ by \Cref{proposition:precedes}, and we take $\tilde{F}=F'$. So, let $\al_s\os \mu$. 
    Then  $F'\in\T_s$ and  $F'\ll F$ by \Cref{proposition:order ll}. Since $\al_s<m-1$, we have $F,F'\notin\cD$. Hence, $F'\prec F$ by \Cref{def:prec} \ref{def 1}. Thus, we take $\tilde{F}=F'$.
        
     Now suppose that $s_1=\al_s+4$. First, let $\mu\neq\al_s+2$. We have $3\leq\al_s<\al_s+4=s_1<n-1$.  It follows that $\mu\nsim\al_s$ or $\mu\nsim s_1$.  Moreover, $\al_s\nsim s_1$, and hence  $F'\in M(\D{3}{C_n^2})$.    
      If $\mu\os\al_s$, then $F'\prec F$ by \Cref{proposition:precedes}, and we take $\tilde{F}=F'$. So assume that $\al_s\os \mu$. 
      Then  $F'\in\T_s$ and  $F'\ll F$ by \Cref{proposition:order ll}. Since $\al_s<m-1$, we have $F,F'\notin\cD$.  Therefore, $F'\prec F$ by \Cref{def:prec} \ref{def 1}, and we take $\tilde{F}=F'$.
      
      Now, let $\mu=\al_s+2$. Let $F''^c:=(F^c\setminus\{s_1\})\sqcup\{\mu\}=\{\al_s, \mu, n-1\}$. Since  $3\leq\al_s<\al_s+2=\mu=s_1-2<s_1<n-1$, it follows that $\al_s\nsim n-1$ and $\mu\nsim n-1$. 
      Therefore $F''\in M(\Delta_3(C_n^2))$.  If $\mu\os\al_s$, then $F''\prec F$ by \Cref{proposition:precedes} and we take $\tilde{F}=F''$. So, let $\al_s\os \mu$. Then $F''\in\T_s$ and thus, $F''^c=\{\al_s\}\sqcup\{\mu,n-1\}$. Since $\mu=\al_s+2<\al_s+4=s_1$, we have  $F''\ll F$  by \Cref{definition:order_ll} \ref{s=t}. Also, $\al_s<m-1$ implies that $F,F''\notin\cD$. Hence, $F''\prec F$  by \Cref{def:prec} \ref{def 1}  and we take $\tilde{F}=F''$.

    Therefore, in each case, we have a facet $\tilde{F}\in M(\Delta_3(C_n^2))$ such that $\tilde{F}\prec F$ and $\tilde{F}^c=(F^c\setminus\{\nu\})\sqcup\{\mu\}$ for some $\nu\in F^c$. Thus, $F$ is a spanning facet.   
\end{proof}

\begin{lemma}
    Suppose that $F\in\cS$. Then $F$ is a spanning facet. \label{lemma: if F in S}
\end{lemma}

\begin{proof}

    Since $F\in\cS$, we have $F\in\T_s$ for some $s$ and $F^c=\{\al_s\}\sqcup\{s_1, n-1\}$. Let $\mathcal{U}=V(C_n^2)\setminus\bigl(\{\al_s-4 \ (\text{mod $n$}), \al_s-3, \al_s-2, \ldots, 2m-\al_s-1\}\cup\{n-1\}\bigr)$.
    
   Let $\mu\in F$. Using \Cref{remark:F is a spanning facet}, it is sufficient to find a facet $\tilde{F}$ such that $\tilde{F} \prec F$ and $\tilde{F}^c =( F^c \setminus \{\nu\}) \sqcup \{\mu\}$ for some $\nu \in F^c$.   Let  $F'^c:=(F^c\setminus\{n-1\})\sqcup\{\mu\}=\{\al_s, s_1, \mu\}$ and  $F''^c:=(F^c\setminus\{s_1\})\sqcup\{\mu\}=\{\al_s, \mu, n-1\}$.
    
    We have three cases: (i) $F$ satisfies \ref{s1} (ii) $F$ satisfies \ref{s2} and (iii) $F$ satisfies \ref{s3}.

    \begin{enumerate}[label=(\roman*)]
    
        \item $F$ satisfies \ref{s1}.
        
     We have $\al_s=3$ and $s_1\in V(C_n^2)\setminus(\{0, 1, 2, 3, \ldots, 2m-4\}\sqcup\{n-1\})$. Observe that $s_1\in\mathcal{U}$.  Therefore, $F$ is a spanning facet by \Cref{spannig:s1s2}.
        
        \item $F$ satisfies \ref{s2}.
        
We have $\al_s\in\{4,5,\ldots,m-2\}$ and    $s_1 \in  V(C_n^2)\setminus  \{\al_s-4, \al_s-3, \al_s-2\}\sqcup\{\al_s, \al_s+1, \ldots, 2m-\al_s-1\}\sqcup\{ n-1\}$.   

If $s_1\neq\al_s-1$, then $s_1\in\mathcal{U}$, and  thus, $F$ is a spanning facet by \Cref{spannig:s1s2}.

        So assume that $s_1=\al_s-1$.  We have $m\leq n-4$ (as $n\geq 9$).  We consider three cases: $\mu<\alpha_s$, $\mu\in\{\al_s+1, \al_s+2\}$,  and $\mu>\al_s+2$. 
        
        First, let  $\mu < \alpha_s$. Then  $\mu<s_1=\al_s-1$  (as $\mu\in F$). Moreover, $\al_s \leq m-2$ implies that  $\al_s\os \mu$. Hence $F''=\{\al_s\}\sqcup\{\mu,n-1\}$. Since $4\leq\al_s\leq m-2\leq n-6$, we have $\al_s\nsim n-1$. Further, $\mu<\al_s$ implies that either $\mu\in\{0, 1\}$, or $2\leq\mu<\al_s$. Hence, $\mu\nsim\al_s$ or $\mu\nsim n-1$. Therefore $F''\in M(\Delta_3(C_n^2))$. Since $\mu<s_1$, we have $F''\ll F$  by \Cref{definition:order_ll} \ref{s=t}. Moreover, $\al_s<m-1$ implies that $F,F''\notin\cD$. Therefore, $F''\prec F$ \Cref{def:prec} \ref{def 1} and we take $\tilde{F}=F''$.
        
        Now, let $\mu\in\{\al_s+1, \al_s+2\}$. We have $4\leq\al_s<\mu\leq\al_s+2\leq m\leq n-4<n-1$. Thus, $\al_s\nsim n-1$ and $\mu\nsim n-1$. Hence $F''\in M(\Delta_3(C_n^2))$.  Since $\al_s\leq m-2$, we have $\mu\os\al_s$. Therefore, $F''\prec F$ by \Cref{proposition:precedes}, and we take $\tilde{F}=F''$.
        
        Lastly, assume that $\mu>\al_s+2$. Then $3\leq\al_s-1=s_1<\al_s<\al_s+2<\mu<n-1$ implies that $\mu\nsim\al_s$ and $\mu\nsim s_1$.  Hence $F'\in M(\Delta_3(C_n^2))$. If $\mu\os\al_s$, then  $F'\prec F$ by \Cref{proposition:precedes}, and we take $\tilde{F}=F'$. Now, if $\al_s\os \mu$, then  $F'\in\T_s$ and  $F'\ll F$ by \Cref{proposition:order ll}. Since $\al_s<m-1$, we have $F,F'\notin\cD$. Hence,  $F'\prec F$ by \Cref{def:prec} \ref{def 1}.  Therefore, we take $\tilde{F}=F'$.

        \item $F$ satisfies \ref{s3}.
        
        We have $F^c=\{\al_s\}\sqcup\{s_1,n-1\}$, where $\al_s\in\{m-1, m, m+1,m+2, \ldots, n-3\}$, $s_1 \in V(C_n^2)\setminus (\{\omega, \omega+1, \ldots,\al_s,\al_s+1,\al_s+2,\al_s+3 \text{ (mod  $n$})\}\sqcup\{y\})$,    $\omega = \min\{2m-\al_s, \al_s-4\}$ and $$y=\begin{cases}
    n-1 &\text{ if } \al_s<n-4,\\
    0 &\text{ if } \al_s=n-4,\\
    1 &\text{ if } \al_s>n-4.  
    \end{cases}$$

        We have $m\geq 5$ (as $n\geq 9$).  First, suppose that $s_1\neq\al_s+4$ (mod $n$). Since $s_1\notin\{\omega, \omega+1, \ldots, \al_s+3 \text{ (mod $n$)}\}$ and $\omega\leq\al_s-4$, it follows that $s_1<\al_s-4$ or $s_1>\al_s+4$ (mod $n$). Thus  $|\al_s-s_1|\geq5$. Now, since $\al_s\geq m-1\geq 4$ and $s_1<n-1$, we get $\mu\nsim \al_s$ or $\mu\nsim s_1$. We have  $\al_s\nsim s_1$. Therefore $F'\in M(\Delta_3(C_n^2))$.
        If $\mu\os\al_s$, then  $F'\prec F$ by \Cref{proposition:precedes}, and we take $\tilde{F}=F'$. Now, if $\al_s\os \mu$,  then  $F'\in\T_s$ and  $F'\ll F$ by \Cref{proposition:order ll}.  Since  $s_1<\al_s-4$ or $s_1>\al_s+4$ (mod $n$), we have $F, F'\notin\cD$.
        Hence $F'\prec F$ by \Cref{def:prec} \ref{def 1}.  Therefore, we take $\tilde{F}=F'$.

        Now, let $s_1=\al_s+4$ (mod $n$). Observe that if $\al_s\geq n-5$, then $\al_s+4\ \text{(mod $n$)}=y$. Therefore, since $s_1\neq y$, we have $\al_s\leq n-6$. This means that $\al_s+4\ \text{(mod $n$)}=\al_s+4$. First, let $\mu\notin\{\al_s+2,\al_s+3\}$. We have $4\leq m-1\leq\al_s<\al_s+4=s_1<n-1$.  It follows that $\mu\nsim\al_s$ or $\mu\nsim s_1$. Further, since $\al_s\nsim s_1$,  we get $F'\in M(\Delta_3(C_n^2))$.

          If $\mu\os\al_s$, then $F'\prec F$ by \Cref{proposition:precedes}, and we take $\tilde{F}=F'$. So assume that $\al_s\os \mu$. Then  $F'\in\T_s$ and  $F'\ll F$ by \Cref{proposition:order ll}.           
        Since $n-1>s_1=\al_s+4$ and $\mu\neq\al_s+3$, we have $F,F'\notin\cD$. Therefore, $F'\prec F$ by \Cref{def:prec} \ref{def 1}, and we take $\tilde{F}=F'$.
      
        Now, let $\mu\in\{\al_s+2, \al_s+3\}$. 
        Observe that since $\al_s\geq m-1$, we have $\al_s\os \mu$, and hence $F''=\{\al_s\}\sqcup\{\mu,n-1\}$. 
         We have $4\leq m-1\leq\al_s<\al_s+2<\al_s+3\leq n-3<n-1$ (as $\al_s\leq n-6$). Clearly, $\al_s\nsim n-1$. Moreover, if $\mu=\al_s+2$, then $\mu\nsim n-1$; and if $\mu=\al_s+3$, then $\mu\nsim\al_s$. Thus $F''\in M(\D{3}{C_n^2})$.
         Since $\mu<s_1$, we have $F''\ll F$. Further, $\al_s\leq n-6$ implies that $n-1\geq \al_s+5$ and thus, $F,F''\notin\cD$. Hence, $F''\prec F$ by \Cref{def:prec} \ref{def 1}, and we take $\tilde{F}=F''$.

        Thus, in each case, we have a facet $\tilde{F}\in M(\Delta_3(C_n^2))$ such that $\tilde{F}\prec F$ and $\tilde{F}^c=(F^c\setminus\{\nu\})\sqcup\{\mu\}$ for some $\nu\in F^c$. Thus, $F$ is a spanning facet.
         \end{enumerate}
    \end{proof}

\begin{proposition}\label{proposition:mu=x+2}
    Suppose $F\in M(\D{3}{C_n^2})$ be such that either $F^c=\{x, x+1\ (\text{mod $n$}), x+4\ (\text{mod $n$})\}$, or $F^c=\{x, x+3 \ (\text{mod $n$}), x+4 \ (\text{mod $n$})\}$ for some $x\in V(C_n^2)$. Then $F$ is not a spanning facet.
\end{proposition}
\begin{proof}
    We have $x+2$ (mod $n$) $\in F$. Let ${F^\nu}^c=(F^c\setminus\{\nu\})\sqcup\{x+2\text{ (mod $n$)}\}$ for $\nu\in F^c$. Then for any $\nu\in F^c$, $C_n^2[{F^\nu}^c]$ is connected. Hence, there does not exist a facet $\tilde{F}$ such that $\tilde{F} \prec F$ and $\tilde{F}\cap F = F\setminus\{x+2\text{ (mod \ $n$)}\}$.  Therefore,  $F$ is not a spanning facet by definition.
\end{proof}

\begin{remark}\label{remark:not in D}
     Observe that  if $F\in\cD$, then either $F^c=\{x,x+1\ (\text{mod \ $n$}),x+4\ (\text{mod \ $n$})\}$ or $F^c=\{x,x+3\ (\text{mod \ $n$}),x+4\ (\text{mod \ $n$})\}$ for some $x\in V(C_n^2)$. Therefore, if $F\in\cD$, then $F$ is not a spanning facet by \Cref{proposition:mu=x+2}.  
\end{remark}

 \begin{proposition} \label{proposition:preceeding violation}
Let $F, F' \in M(\Delta_3(C_n^2))$. Let  $F \in \T_s$ and $F' \in \T_t$ for some $s\neq t$ such that  ${F}^c = \{\al_s\} \sqcup \{s_1, s_2\}$ and ${F'}^c = \{\mu, s_1, s_2\}$ for some  $\mu \in F$. If $\al_s\os\mu$, then $F\prec F'$. 
\end{proposition}

\begin{proof}
     It follows that $F^c=(F'^c\setminus\{\mu\})\sqcup\{\al_s\}$ and  since $F'\in\T_t$, $\al_t\in\{\mu,s_1,s_2\}$. 
     Then,  $\al_s\os\mu,s_1,s_2$ implies that  $F\prec F'$ by \Cref{proposition:precedes}.
 \end{proof}

  \begin{proposition}\label{proposition:al_s and s_2} 
  Let $F\in M(\D{3}{C_n^2})$ be a spanning facet such that $F\in\T_s$ for some  $s$ and $F^c=\{\al_s\}\sqcup\{s_1,s_2\}$. Then $\al_s\notin\{0,1,2,n-2,n-1\}$ and $s_2=n-1$.
\end{proposition}

\begin{proof}
     By \Cref{remark:range of al_s},   $\al_s\notin\{0,1,n-1\}$. We have $C_n^2[F^c]$ is disconnected and by \Cref{remark:Aj}, $\al_s\os s_1,s_2$. Therefore, since $F$ is a spanning facet, we get $\al_s\notin\{2,n-2\}$ by \Cref{proposition:mu=x+2}. Hence $\al_s\notin\{0,1,2,n-2,n-1\}$.
    
    Observe that $\al_s\os n-1$. 
    Suppose $n-1\notin F ^c$.  Then \Cref{remark:F is a spanning facet} implies that there exists $ \tilde{F} \in M(\Delta_3(C_n^2))$ such that $\tilde{F} \prec F$ and $\tilde{F}^c=(F^c\setminus\{\nu\})\sqcup\{n-1\}$ for some $\nu\in F^c$.
    
    If $\nu=\al_s$, then  by \Cref{proposition:preceeding violation}, $\al_s\os n-1$ implies that $F \prec \tilde{F}$, a contradiction. So,  $\nu\in\{s_1,s_2\}$. Since $\al_s\os s_1,s_2,n-1$ and $s_1<s_2<n-1$,  we have either $\tilde{F} = \{\al_s\}\sqcup\{s_2,n-1\}$ or $\tilde{F} =\{\al_s\}\sqcup\{s_1,n-1\}$. This implies that $F\in\T_s$ and $F \ll \tilde{F}$.  
    Since $F$ is a spanning facet,  $F \notin\cD$ by  \Cref{remark:not in D}.
    Therefore, $F \prec \tilde{F}$ by \Cref{def:prec} \ref{def 1} and \ref{def 3}, again a contradiction. Thus, $n-1\in F ^c$. 

    Clearly, $\al_s\os n-1$ and $s_1<s_2$ implies that $s_2=n-1$.
\end{proof}

\begin{proposition}\label{proposition:F is not a spanning facet}
    Let $F \in M(\Delta_3(C_n^2))$ such that  $F \in \T_s$ and  ${F}^c = \{\al_s\} \sqcup \{s_1, s_2\}$. Suppose there exists  $\mu \in F$ such that $\al_s\os\mu$ and $s_1<\mu$.  Let ${F'}^c = ({F}^c \setminus \{s_2\}) \sqcup \{\mu\}$. If either $F'\notin M(\D{3}{C_n^2})$, or  $F'\in M(\Delta_3(C_n^2))$ such that $F'\in\cD$, then $F$ is not a spanning facet.
\end{proposition}

\begin{proof}
    Assume that either $F'\notin M(\D{3}{C_n^2})$, or $F'\in\cD$ whenever $F'\in M(\Delta_3(C_n^2))$.
    Suppose $F$ is a spanning facet.  We have  $\mu \in F$.  By \Cref{remark:F is a spanning facet}, there exists $ \tilde{F} \in M(\Delta_3(C_n^2))$ such that $\tilde{F} \prec F$ and $\tilde{F}^c=(F^c\setminus\{\nu\})\sqcup\{\mu\}$ for some $\nu\in F^c$. Moreover, $F\notin\cD$ by  \Cref{remark:not in D}, and $s_2=n-1$ by \Cref{proposition:al_s and s_2}.  
    
    If $\nu=\al_s$, then $\al_s\os\mu$ implies that $F\prec\tilde{F}$ by \Cref{proposition:preceeding violation}, a contradiction. Now,  if $\nu=s_1$, then $\al_s\os\mu$ implies that  $\tilde{F}^c=\{\al_s\}\sqcup\{\mu,n-1\}$.
    It follows that $\tilde{F}\in\T_s$, and since $s_1<\mu$, $F\ll \tilde{F}$. Hence $F\notin\cD$ implies that $F\prec \tilde{F}$ by \Cref{def:prec} \ref{def 1} and \ref{def 3}, again a contradiction. So, $\nu=s_2=n-1$. Then $\tilde{F}=F'$.

    If $F' \notin M(\Delta_3(C_n^2))$, then this implies that $\tilde{F}\notin M(\Delta_3(C_n^2))$, which is a contradiction.

    If $F'\in M(\D{3}{C_n^2})$, then by our assumption, $F'\in\cD$. Now, since $\al_s\os\mu$ and $s_1<\mu$, we get  $F'^c=\{\al_s\}\sqcup\{s_1,\mu\}$ and thus $F'\in\T_s$. Further, since $F\notin\cD$, $F\prec F'$ by \Cref{def:prec} \ref{def 3}, which is a contradiction. 
    
  Thus $F$ is not a spanning facet. 
\end{proof}

\begin{lemma}\label{lemma:Spanning facet is in S}
    Let $F\in M(\D{3}{C_n^2})$ be a spanning facet. Then $F\in\cS$. \label{lemma: then F in S}
\end{lemma}

\begin{proof}    
 Let $F\in\T_p$ such that $F^c=\{\al_p\}\sqcup\{p_1,p_2\}$.
   From \Cref{proposition:al_s and s_2}, $\al_p\notin\{0,1,2,n-2,n-1\}$ and $p_2=n-1$. Hence $\al_p\in\{3,4,\ldots,n-3\}$. For each $\al_p$, we find the possible values of $p_1$ for which $F$ is a spanning facet and show that $F\in\cS$ for all such values of $p_1$.  We have $5\leq m\leq n-4$ (as $n\geq 9$).
We consider the following five cases based on the values of $\al_p$.

\begin{enumerate}[label=(\roman*)]
\item $\al_p=3$.

We show that $F$ satisfies \ref{s1}. For this, we need to show that $p_1\in V(C_n^2)\setminus(\{0,1,2,3,4,\ldots,2m-4\}\sqcup\{n-1\})$.
 We have $\al_p=3<m$ (as $m\geq 5$) and $\al_p\os p_1$. By \Cref{remark:order in S} \ref{part:i}, either $p_1<3$ or $p_1\geq 2m-3$. Thus $p_1\notin\{3,4,\ldots,2m-4\}$.

 Suppose $p_1<3$. Then $C_n^2[F^c]$ is disconnected implies that $p_1\in\{0,2\}$. It follows that either $F^c=\{n-1,(n-1)+1\ (\text{mod \ $n$}),(n-1)+4\ (\text{mod \ $n$})\}$ or $F^c=\{n-1,(n-1)+3\ (\text{mod \ $n$}),(n-1)+4\ (\text{mod \ $n$})\}$, which is a contradiction to \Cref{proposition:mu=x+2}.
Hence, $p_1\geq 2m-3$. Since $p_1<p_2=n-1$, $p_1\neq n-1$. This means that $p_1\in V(C_n^2)\setminus(\{0,1,2,3,4,\ldots,2m-4\}\sqcup\{n-1\})$.

   \item $\al_p\in\{4,5,\ldots,m-2\}$.
   
   In this case, we show that  $p_1\in V(C_n^2)\setminus(\{\al_p-4,\al_p-3,\al_p-2\}\sqcup\{\al_p,\al_p+1,\ldots,2m-\al_p-1\}\sqcup\{n-1\})$, which implies that $F$ satisfies \ref{s2}.
   
    Since $\al_p\leq m-2$  and $\al_p\os p_1$, either $p_1<\al_p$ or $p_1\geq 2m-\al_p$ by \Cref{remark:order in S} \ref{part:i}. Thus $p_1\notin\{\al_p,\al_p+1,\ldots,2m-\al_p-1\}$. Also,  $p_1<p_2=n-1$ implies that $p_1\neq n-1$.

    Suppose $p_1\in\{\al_p-4,\al_p-3\}$. We have $p_1+2\leq\al_p-1\leq m-3<n-1=p_2$. Thus $p_1+2\in F$.    
     Clearly, $p_1<p_1+2$. Moreover, observe that $\al_p\os p_1+2$, and $C_n^2[(\di{F^c}{p_2})\sqcup\{p_1+2\}]=C_n^2[\{p_1,p_1+2,\al_p\}]$ is connected. Thus $F'\notin M(\Delta_3(C_n^2))$, where $F'^c=(F^c\setminus\{p_2\})\sqcup\{p_1+2\}$. It follows from \Cref{proposition:F is not a spanning facet} that $F$ is not a spanning facet, a contradiction.  Therefore $p_1\notin\{\al_p-4, \alpha_p-3\}$.
     
    Now, suppose  $p_1=\al_p-2$. Since $C_n^2[F^c]$ is disconnected, $p+1 = \alpha_p-1 \in F$. We have $\al_p\os \alpha_p-1 = p+1$ and $p_1<p_1+1$.  Moreover, $C_n^2[(\di{F^c}{p_2})\sqcup\{p_1+1\}]=C_n^2[\{\al_p-2,\al_p-1,\al_p\}]$ is connected.
    Hence, using \Cref{proposition:F is not a spanning facet}, $F$ is not a spanning facet, a contradiction. Thus, $p_1\neq \al_p-2$.

    Therefore, $p_1\in V(C_n^2)\setminus(\{\al_p-4,\al_p-3,\al_p-2,\al_p,\al_p+1,\ldots,2m-\al_p-1\}\sqcup\{n-1\})$. 

      \item  $\al_p\in\{m-1,m\}$.
      
      Let $$y:=\begin{cases}
     n-1 &\text{ if }\al_p<n-4,\\
     0 &\text{ if  }\al_p=n-4,\\
     1 &\text{ if }\al_p>n-4.
 \end{cases}$$
    We  show that $p_1\in V(C_n^2)\setminus(\{\al_p-4,\al_p-3,\al_p-2,\al_p-1,\al_p,\al_p+1,\al_p+2,\al_p+3\ (\text{mod \ $n$})\}\sqcup\{y\})$,  which implies that $F$ satisfies \ref{s3} (as $\min\{2m-\al_p,\al_p-4\}=\al_p-4$). 
    
    We have $F^c=\{\al_p\}\sqcup\{p_1,p_2\}$, where $p_2=n-1$. Clearly, $p_1\notin\{\al_p,n-1\}$.

     Suppose $p_1=y$. We have $\al_p\leq m\leq n-4$. So $y\neq1$. Since $p_1\neq n-1,$ we get $y=0$.  Clearly, if $\al_p\neq n-4$, then $y\neq0$. Thus $\al_p=n-4$. This means that
    $n=9$ and $\al_p=m$, thereby implying that $F^c=\{m,m+3,m+4\ (\text{mod $n$)}\}$. Then $F$ is not a spanning facet by \Cref{proposition:mu=x+2}.  This is a contradiction. So $p_1\neq y$.

     We now show that if $p_1$ $\in \{\al_p-4,\al_p-3,\al_p-2,\al_p-1,\al_p+1,\al_p+2\}$, then there exists  $\mu \in F$ such that $\alpha_p \os \mu$, $p_1 < \mu$ and $C_n^2[(F^c \setminus \{p_2\}) \sqcup \{\mu\}]$ is connected. 
     Then \Cref{proposition:F is not a spanning facet} implies that $F$ is not a spanning facet, which is a contradiction. 
     
    \begin{enumerate}
    	\item $p_1\in\{\al_p-4,\al_p-3\}$. 
    	
    	We have  $p_1+2\leq \al_p-1\leq m-1<n-1=p_2$. It follows that $p_1+2\in F$.   Take,  $\mu = p_1+2$. Then   $\al_p\os \mu$, $p_1<\mu$ and  $C_n^2[(F^c \setminus \{p_2\}) \sqcup \{\mu\}] = C_n^2[\{p_1, p_1+2, \alpha_p \}] $ is connected. 
    	
    	\item $p_1 = \alpha_p-2$.
    	
    	Since $p_1+1= \al_p-1\leq m-1<n-1=p_2$, we have $p_1+1\in F$. Take $\mu = p_1+1$.  Then   $\al_p\os \mu$, $p_1<\mu$ and  $C_n^2[(F^c \setminus \{p_2\}) \sqcup \{\mu\}]  = C_n^2[\{\alpha_p-2, \alpha_p-1, \alpha_p \}]$ is connected. 
    	
    		\item $p_1 = \alpha_p-1$.
    	
        Here,  $p_1+3=\al_p+2\leq m+2<n-1=p_2$, which implies that $p_1+3\in F$. Take $\mu = p_1+3$. Then   $\al_p\os \mu$,  $p_1<\mu$ and  $C_n^2[(F^c \setminus \{p_2\}) \sqcup \{\mu\}]  = C_n^2[\{\alpha_p-1, \alpha_p, \alpha_p+2 \}]$ is connected. 
    	
    		\item $p_1 = \alpha_p+1$.
    	
    	In this case, $\al_p\os p_1$ implies that $\al_p=m$. Then  $p_1+1=m+2<n-1$. So $p_1+1\in F$.  Take, $\mu = p_1+1$.  We have $\al_p=m\os m+2=p_1+1 = \mu $ and $p_1<\mu$. Clearly, $C_n^2[(F^c\setminus\{p_2\})\sqcup\{\mu\}]  = C_n^2[\{\alpha_p, \alpha_p+1, \alpha_p+2\}]$ is connected.   
    	
    	\item   $p_1=\al_p+2$. 
    	
    	Here, since $C_n^2[F^c]$ is disconnected, we have $n\geq 10$. It follows that $p_1+1=\al_p+3\leq m+3<n-1$, and thus $p_1+1\in F$. Take $\mu = p_1+1$.    
    	Observe that $\al_p\os \mu $ and $C_n^2[(F^c\setminus\{p_2\})\sqcup\{\mu\}] = C_n^2[\{\alpha_p, \alpha_p+2, \alpha_p+3\}]$ is connected.
    	
    	\end{enumerate}

Thus, $p_1 \notin \{\al_p-4,\al_p-3,\al_p-2,\al_p-1,\al_p,\al_p+1,\al_p+2\} \sqcup\{y\}$.

    Now, suppose    $p_1=\al_p+3\ (\text{mod \ $n$})$. 
    	Since $\al_p\leq m$, we have $\al_p+3\ (\text{mod \ $n$})=\al_p+3$. Suppose $n=9$. Then  $m+3=n-1$. Therefore, if $\al_p=m-1$, then $F^c=\{\al_p,\al_p+3,\al_p+4\}$, which implies that $F$ is not a spanning facet by  \Cref{proposition:mu=x+2}. This is a contradiction.  So $\al_p=m$.  Then $p_1=m+3=n-1=p_2$, which is again a contradiction.  
     
        Thus, we assume that $n\geq 10$. It follows that $p_1+1=\al_p+4\leq m+4\leq n-1$. If $p_1+1=n-1$,  then $\al_p=m$ and $F^c=\{m,m+3,m+4\}$. By  \Cref{proposition:mu=x+2}, it follows that $F$ is not a spanning facet, a contradiction. Hence $p_1+1<n-1$ and thus, $p_1+1\in F$. We have $\al_p\os p_1+1$ and $p_1<p_1+1$. Let $F'^c:= (F^c \setminus \{p_2\}) \sqcup \{p_1+1\}$.  If $F'\notin M(\D{3}{C_n^2})$, then  \Cref{proposition:F is not a spanning facet} implies that $F$ is not a spanning facet, which is a contradiction.  So $F'\in M(\D{3}{C_n^2})$. Then  $F'^c=\{\al_p\}\sqcup\{\al_p+3,\al_p+4\}$. This means that $F'$  satisfies \ref{d4}, and thus $F'\in\cD$. Hence,  \Cref{proposition:F is not a spanning facet} implies that $F$ is not a spanning facet, again a contradiction.
       So $p_1\neq\al_p+3\text{ (mod \ $n$})$.     	

 Therefore, $p_1\in V(C_n^2)\setminus(\{\al_p-4,\al_p-3,\al_p-2,\al_p-1,\al_p,\al_p+1,\al_p+2,\al_p+3\ (\text{mod \ $n$})\}\sqcup\{y\})$.

\item $\al_p=m+1$.

Let $$y:=\begin{cases}
    n-1 &\text{ if } m+1<n-4,\\
    0 &\text{ if } m+1=n-4,\\
    1 &\text{ if } m+1>n-4.  
    \end{cases}$$
We show that $F$ satisfies \ref{s3}. For this, we show that $p_1\in V(C_n^2)\setminus(\{m-3,m-2,m-1,m,m+1,m+2,m+3,m+4\ (\text{mod $n$})\}\sqcup\{y\})$ (as $\min\{2m-\al_p,\al_p-4\}=\al_p-4=m-3$). 
We have $F^c=\{m+1\}\sqcup\{p_1,p_2\}$, where $p_2=n-1$.

Suppose  $p_1=y$. Since $p_1\neq p_2= n-1$, we get $y\in\{0,1\}$. If $y=0$, then $m+1=n-4$. Hence $p_2=n-1=m+4$, which implies that $F^c=\{m+1,(m+1)+3\text{ (mod \ $n$}),(m+1)+4\text{ (mod \ $n$})\}$. Thus, $F$ is not a spanning facet by \Cref{proposition:mu=x+2}.
This is a contradiction.  So $y=1$,  and  then $m+1>n-4$. This implies that $C_n^2[F^c]$ is connected, again a contradiction. Hence $p_1\neq y$.

Since $\al_p=m+1\os p_1$, we have $p_1\notin\{m-1,m,m+1\}$.   Suppose $p_1 \in \{ m-3, m-2, m+2, m+3, m+4\ (\text{mod $n$})\}$.

\begin{enumerate}
	\item  $p_1=m-3$.
	
	Here, $p_1+1=m-2\in F$. We have $\al_p=m+1\os m-2=p_1+1$ and $p_1<p_1+1$.  Observe that if $F'^c=\{m+1\}\sqcup\{m-3,m-2\}$, then $F'\in M(\D{3}{C_n^2})$ such that $F'$  satisfies  \ref{d3} and thus, $F'\in\cD$. Therefore,  \Cref{proposition:F is not a spanning facet} implies that $F$ is not a spanning facet, a contradiction. Hence $p_1\neq m-3$.

\item  $p_1=m-2$. 

We have $p_1+4=m+2\in F$. Also, $\al_p=m+1\os m+2=p_1+4$ and $p_1<p_1+4$. If $F'^c=\{m+1\}\sqcup\{m-2,m+2\}$, then $F'\in M(\D{3}{C_n^2})$ such that $F'$ satisfies  \ref{d1}.  Thus $F'\in\cD$. Therefore, by  \Cref{proposition:F is not a spanning facet}, $F$ is not a spanning facet. This is a contradiction. Hence $p_1\neq m-2$.

\item $p_1\in\{m+2,m+3\}$.

 Observe that if $n=9$, then $n-1=m+3$; if $n\in\{10,11\}$, then $n-1=m+4$; and if $n\in\{12,13\}$, then $n-1=m+5$. We have $F^c=\{m+1\}\sqcup\{p_1,n-1\}$. Since $C_n^2[F^c]$ is disconnected and $p_1<n-1$,  it follows that  if $p_1=m+2$, then $n\geq 12$; and if $p_1=m+3$, then $n\geq 14$.  Therefore, in each case, we have $p_1+2<n-1$. This implies that $p_1+2\in F$.  We have $\al_p\os p_1+2$, $p_1<p_1+2$ and    
  $C_n^2[(\di{F^c}{n-1})\sqcup\{p_1+2\}]=C_n^2[\{\al_p,p_1,p_1+2\}]$ is connected. Using   \Cref{proposition:F is not a spanning facet}, we get a contradiction to the fact that $F$ is a spanning facet. Hence  $p_1\notin\{m+2,m+3\}$.

\item  $p_1=m+4\ (\text{mod $n$})$.

Then $F^c=\{m+1,m+4\ (\text{mod $n$}),n-1\}$. Since $C_n^2[F^c]$ is disconnected and $p_1<n-1$, we get $n\geq 12$. Further, if $n\in\{12,13\}$, then $n-1=m+5$, and thus, $F^c=\{m+1,(m+1)+3,(m+1)+4\}$. 
It follows from \Cref{proposition:mu=x+2} that $F$ is not a spanning facet. This is a contradiction. So, assume that $n\geq 14$. Then $p_1=m+4\ (\text{mod $n$})=m+4$ and $p_1+1=m+5<n-1$. Thus $p_1+1\in F$. Let $F'^c:=\{m+1\}\sqcup\{m+4,m+5\}$. Then $F'\in M(\D{3}{C_n^2})$ and satisfies \ref{d4}. Hence $F'\in\cD$.  Using   \Cref{proposition:F is not a spanning facet}, it follows that $F$ is not a spanning facet, a contradiction. So $p_1\neq m+4\ (\text{mod $n$})$ 

	\end{enumerate}

Therefore $p_1\not\in \{m-3,m-2,m-1,m,m+1,m+2,m+3,m+4\ (\text{mod $n$})\}\sqcup\{y\}$.

\item $\al_p\in\{m+2,\ldots,n-3\}$.

We show that $F$ satisfies \ref{s3}. Since $\min\{2m-\al_p,\al_p-4\}=2m-\al_p$, we prove that $p_1\in V(C_n^2)\setminus(\{2m-\al_p,\ldots,\al_p,\al_p+1,\al_p+2,\al_p+3 \text{ (mod \ $n$})\}\sqcup\{y\})$, where $$y=\begin{cases}
    n-1 &\text{ if } \al_p<n-4,\\
    0 &\text{ if } \al_p=n-4,\\
    1 &\text{ if } \al_p>n-4.  
    \end{cases}$$

      We have $\al_p\geq m+2$  and $\al_p\os p_1$. Thus, either $p_1< 2m-\al_p$ or $p_1>\al_p$ by \Cref{remark:order in S} \ref{part:ii}. This means that $p_1\notin\{2m-\al_p,\ldots,\al_p-1,\al_p\}$.  Hence, it is sufficient to show that  $p_1 \not\in \{\al_p+1,\al_p+2,\al_p+3 \text{ (mod \ $n$})\}\sqcup\{y\}$.
     
     \begin{enumerate}
     	\item $\al_p\in\{m+2,m+3,\ldots,n-6\}$.
     	
         Here, since $\al_p\leq n-6<n-4$,  we have $y=n-1$. Clearly, $p_1\neq y$. Suppose $p_1\in\{\al_p+1,\al_p+2\}$. Then $p_1+1\leq \al_p+3\leq n-3<n-1$, and thus, $p_1+1\in F$.  We have $\al_p\os p_1+1$ and $p_1<p_1+1$. Further, since   $C_n^2[(\di{F^c}{p_2})\sqcup\{p_1+1\}]=C_n^2[\{\al_p,p_1,p_1+1\}]$ is connected,  $F$ is not a spanning facet by \Cref{proposition:F is not a spanning facet}. This is a contradiction.  Therefore $p_1\notin\{\al_p+1,\al_p+2\}$.
     	
         Suppose $p_1=\al_p+3\text{ (mod \ $n$})$. We have $\al_p+3\leq n-3<n-1$. So $\al_p+3\ (\text{mod \ $n$})=\al_p+3$. Then $p_1+1=\al_p+4\leq n-2< n-1$, and thus, $p_1+1\in F$.  Clearly, $\al_p\os p_1+1$ and $p_1<p_1+1$.
        If  $F'^c=\{\al_p\}\sqcup\{\al_p+3,\al_p+4\}$, then  $F'\in M(\D{3}{C_n^2})$ and satisfies \ref{d4}. Thus $F'\in\cD$. By \Cref{proposition:F is not a spanning facet}, it follows that $F$ is not a spanning facet, a contradiction. Therefore, $p_1\neq \al_p+3\text{ (mod \ $n$})$.

  \item  $\al_p=n-5$.

   We have $y=n-1$ and $\al_p+3\ (\text{mod \ $n$})=\al_p+3$. Further, $\al_p+4=n-1=p_2$. Therefore, since $C_n^2[F^c]$ is disconnected and $p_1<p_2$, it follows that $p_1\notin\{\al_p+2,y\}$.
   Moreover, if $p_1\in\{\al_p+1,\al_p+3\}$, then either $F^c=\{\al_p,\al_p+1,\al_p+4\}$ or $F^c=\{\al_p,\al_p+3,\al_p+4\}$.  Then, by \Cref{proposition:mu=x+2}, it follows that $F$ is not a spanning facet, a contradiction. Hence $p_1\notin\{\al_p+1,\al_p+3\}$. Therefore, $p_1\not\in \{\al_p+1,\al_p+2,\al_p+3\ (\text{mod \ $n$})\}\sqcup\{y\})$.

  \item $\al_p=n-4$.

   Then $y=0$ and $\al_p+3\ (\text{mod \ $n$})=\al_p+3=n-1=p_2$. If $p_1=y$, then $F^c=\{\al_p,\al_p+3,\al_p+4\ (\text{mod \ $n$})\}$, which contradicts the fact that $F$ is a spanning facet by  \Cref{proposition:mu=x+2}. Hence $p_1\neq y$.  Further, since $C_n^2[F^c]$ is disconnected and $p_1<p_2$, we get $p_1\notin\{\al_p+1,\al_p+2,\al_p+3\}$. 
    Therefore, $p_1\not\in \{\al_p+1,\al_p+2,\al_p+3\ (\text{mod \ $n$})\}\sqcup\{y\})$.

  \item  $\al_p=n-3$.

  We have $y=1$. Since $C_n^2[F^c]$ is disconnected and $p_1<p_2$, it follows that $p_1\notin\{\al_p+1,\al_p+2,\al_p+3\ (\text{mod \ $n$}),1)\}$.  
 Therefore, $p_1\not\in \{\al_p+1,\al_p+2,\al_p+3\ (\text{mod \ $n$})\}\sqcup\{y\})$.

     \end{enumerate}

\end{enumerate}
This completes the proof of  \Cref{lemma:Spanning facet is in S}.
\end{proof}

\begin{proof}[Proof of \Cref{theorem:main}]
 In \Cref{subsection:shelling order}, we have proved that $\prec$ gives a shelling order for $\D{3}{C_n^2}$. Hence the $3$-cut complex $\D{3}{C_n^2}$ is shellable  for  $ n \geq 9$.

Now, using Lemmas \ref{lemma: if F in S} and \ref{lemma: then F in S}, we see that $F\in M(\D{3}{C_n^2})$ is a spanning facet for the shelling order given by $\prec$ on $\D{3}{C_n^2}$ if and only if $F\in\cS$. Therefore, to count the number of spanning facets, it is enough to count the total facets in $\cS$.  Let $F\in\cS$. Then $F\in \T_s$ for some $s$ such that $F^c=\{\al_s\}\sqcup\{s_1,n-1\}$ and $F$ satisfies one of the conditions from \ref{s1} to \ref{s3}.

 If $F$ satisfies \ref{s1}, then $\al_s=3$ and there are $n-(2m-2)=n-2m+2$ possibilities for $s_1$ and hence the same number of spanning facets satisfying \ref{s1}.
 
 If $F$ satisfies \ref{s2}, then there are $m-5$ possibilities for $\al_s$, and for each value of $\al_s$, there are $n-(2m-2\al_s+4)$ possibilities for $s_1$. Therefore, the number of spanning facets satisfying \ref{s2}
 \begin{align*}
  &=(m-5)(n-2m-4)+\left(2\sum_{\al_s=4}^{m-2}\al_s\right)\\
 &=(m-5)(n-2m-4)+(m-2)(m-1)-12\\
    &= mn-m^2+3m-5n+10.
\end{align*}
 
 Suppose $F$ satisfies \ref{s3}. If $\al_s\in\{m-1,m,m+1\}$, then $\omega=\al_s-4$. For each $\al_s$, there are $n-9$ possibilities for $s_1$.  If $\al_s\in\{m+2, \ldots, n-3\}$, then $\omega=2m-\al_s$. There are $n-m-4$ possibilities for $\al_s$, and for each value of $\al_s$, there are $n-(2\al_s-2m+5)$ possibilities for $s_1$. Thus, the number of spanning facets satisfying \ref{s3}
 \begin{align*}
  &= 3(n-9)+(n-m-4)(n+2m-5)-\left(2\sum_{\al_s=m+2}^{n-3}\al_s\right)\\
 &= 3(n-9)+(n-m-4)(n+2m-5)-(n-3)(n-2)+(m+1)(m+2)\\
    &=   mn-m^2-n-11.
\end{align*}
 
 Therefore,  the total  number of spanning facets
\begin{align*}
  &=(n-2m+2)+(mn-m^2+3m-5n+10)+ mn-m^2-n-11\\
    &= 2mn-2m^2+m-5n+1\\
    &= \frac{n^2-9n+2}{2}\\
    &=\binom{n-4}{2}-9.
\end{align*}

 Hence from \Cref{theorem:wedge}, we get  $
 \Delta_3(C_n^2) \simeq \bigvee\limits_{\binom{n-4}{2}-9} \mathbb{S}^{n-4}.
 $
\end{proof}

\section{Conclusion and Future Directions}\label{section:future_directions}

 In \cite{Bayer2024Cutcomplex} and \cite{Bayer2024TotalCutcomplex}, the authors completely determined the homotopy type of the complex $\Delta_k(C_n)$ for all $k$. They also proved that the complex $\Delta_2(C_n^2)$ (if nonvoid) is not shellable. We have proved that the complex $\Delta_3(C_n^2)$ of the squared cycle graphs is shellable for all $n\ge 9$, by constructing a shelling order.  We believe that this shelling order can be generalized to prove that the complex $\Delta_k(C_n^2)$ is shellable for all  $k\geq 4$ and $n\ge k+6$ (\Cref{conjecture}).







We further investigated the $k$-cut complex $\Delta_k(C_n^p)$ for powered cycle graphs $C_n^p$ and, using SageMath, computed their homology groups (with coefficient $\Z$).   
We have data for some small values of $k$, $p$ and $n$, which is illustrated in Tables \ref{tab:2-cut}-\ref{tab:6-cut}. Missing entries in these tables indicate that the corresponding complex $\Delta_k(C_n^p)$ is void. An entry of the form $i: \Z^{\beta}$ denotes that the $i$-th homology group is $\Z^{\beta}$.

\begin{table}[h!]
    \tiny
    \centering
    \begin{tabular}{|c|c|c|c|c|c|c|c|c|c|c|c|}
     \hline
    \diagbox[linewidth=0.2pt, width=.5cm, height=.5cm]{$p$}{$n$}  &8 &9 &10 &11 &12 &13 &14 &15 &16\\
    \hline
      3 &$2: \mathbb{Z}^{}$ &$4: \mathbb{Z}^{2}$& $6: \mathbb{Z}^{}$& $7: \mathbb{Z}^{}$& $8: \mathbb{Z}^{}$& $9: \mathbb{Z}^{}$ & $10: \mathbb{Z}^{}$ & $11: \mathbb{Z}^{}$ & $12: \mathbb{Z}^{}$\\
    \hline
      4 & && $3: \mathbb{Z}^{}$& $5: \mathbb{Z}^{}$& $7: \mathbb{Z}^{3}$& $9: \mathbb{Z}^{}$ & $10: \mathbb{Z}^{}$ & $11: \mathbb{Z}^{}$ & $12: \mathbb{Z}^{}$\\
    \hline
      5 & && & & $4: \mathbb{Z}^{}$& $7: \mathbb{Z}^{}$ & $8: \mathbb{Z}^{}$ & $10: \mathbb{Z}^{4}$ & $12: \mathbb{Z}^{}$\\
    \hline
      6  & && & &  & & $5: \mathbb{Z}^{}$ & $8: \mathbb{Z}^{2}$ & $10: \mathbb{Z}^{}$\\
    \hline
    \end{tabular}
    \caption{ Homology  of $2$-cut complexes $\Delta_{2}(C_n^p)$}
    \label{tab:2-cut}
\end{table}

Note that the nonvoid $2$-cut complexes $\Delta_2(C_n^p)$ in \Cref{tab:2-cut}, have non trivial homology  in  dimension lower than the dimension of $\Delta_2(C_n^p)$. We propose the following conjecture.

\begin{conjecture}\label{conjecture:2-cut}
For $p \geq 3$, if $\Delta_2(C_n^p) \neq \emptyset$, then $\Delta_2(C_n^p)$ is not shellable.
\end{conjecture}

A natural question arises here:  What is the homotopy type of $\Delta_2(C_n^p)$ for $p\geq 3$? 

    

\begin{table}[h!]
    \centering\tiny
    \begin{tabular}{|c|c|c|c|c|c|}
    \hline
    \diagbox[linewidth=0.2pt, width=.85cm, height=.7cm]{$n$}{$p$} &2 &3 &4 &5 & 6 \\
    \hline
        7 &$1:\mathbb{Z}^{}$&&&&\\
    \hline
        8  &$3:\mathbb{Z}^{}$&&& &\\
    \hline
        9  &$5:\mathbb{Z}^{}$&$3:\mathbb{Z}^{}$&&&\\
    \hline
        10  &$6:\mathbb{Z}^{6}$&$4:\mathbb{Z}^{11}$&&&\\
    \hline
        11 &$7:\mathbb{Z}^{12}$&$5:\mathbb{Z}^{}$&$3:\mathbb{Z}^{}$&& \\
    \hline
        12 &$8:\mathbb{Z}^{19}$ &$7:\mathbb{Z}^{2}$ &$6:\mathbb{Z}^{6}$ &&\\
    \hline
        13 &$9:\mathbb{Z}^{27}$ &$9:\mathbb{Z}^{}$&$7:\mathbb{Z}^{40}$ &$5:\mathbb{Z}^{}$&\\
    \hline
        14 &$10:\mathbb{Z}^{36}$ &$10:\mathbb{Z}^{8}$&$8:\mathbb{Z}^{15}$ &$6:\mathbb{Z}^{}$&\\
    \hline
        15  &$11:\mathbb{Z}^{46}$ &$11:\mathbb{Z}^{16}$ &$9:\mathbb{Z}^{}$ &$9:\mathbb{Z}^{21}$ & $5:\mathbb{Z}^{}$\\
    \hline
        16  &$12:\mathbb{Z}^{57}$ &$12:\mathbb{Z}^{25}$ &$11:\mathbb{Z}^{3}$ &$10:\mathbb{Z}^{101}$ & $8:\mathbb{Z}^{31}$\\
    \hline
        17 &$13:\mathbb{Z}^{69}$ &$13:\mathbb{Z}^{35}$ &$13:\mathbb{Z}^{}$ &$11:\mathbb{Z}^{52}$ &$9:\mathbb{Z}^{}$\\
    \hline
        18  &$14:\mathbb{Z}^{82}$ &$14:\mathbb{Z}^{46}$ &$14:\mathbb{Z}^{10}$ &$12:\mathbb{Z}^{19}$ &$12:\mathbb{Z}^{55}$\\
    \hline
        19  &$15:\mathbb{Z}^{96}$ &$15:\mathbb{Z}^{58}$ &$15:\mathbb{Z}^{20}$  &$13:\mathbb{Z}^{}$ &$13:\mathbb{Z}^{210}$\\
    \hline
        20  &$16:\mathbb{Z}^{111}$ &$16:\mathbb{Z}^{71}$ &$16:\mathbb{Z}^{31}$  &$15:\mathbb{Z}^{4}$ &$14:\mathbb{Z}^{126}$\\
    \hline
        21  &$17:\mathbb{Z}^{127}$ &$17:\mathbb{Z}^{85}$ &$17:\mathbb{Z}^{43}$  &$17:\mathbb{Z}^{}$ &$15:\mathbb{Z}^{64}$\\
    \hline
        22  &$18:\mathbb{Z}^{144}$ &$18:\mathbb{Z}^{100}$ &$18:\mathbb{Z}^{56}$  &$18:\mathbb{Z}^{12}$ &$16:\mathbb{Z}^{23}$\\
    \hline
        23  &$19:\mathbb{Z}^{162}$ &$19:\mathbb{Z}^{116}$ &$19:\mathbb{Z}^{70}$  &$19:\mathbb{Z}^{24}$ &$17:\mathbb{Z}^{}$\\
    \hline
        24  &$20:\mathbb{Z}^{181}$ &$20:\mathbb{Z}^{133}$ &$20:\mathbb{Z}^{85}$  &$20:\mathbb{Z}^{37}$ &$19:\mathbb{Z}^{5}$\\
    \hline
        25  &$21:\mathbb{Z}^{201}$ &$21:\mathbb{Z}^{151}$ &$21:\mathbb{Z}^{101}$  &$21:\mathbb{Z}^{51}$ &$21:\mathbb{Z}^{}$\\
    \hline
    \end{tabular}
    \caption{Homology of $3$-cut complexes $\Delta_{3}(C_n^p)$}
    \label{tab:3-cut}
\end{table}

        
   

\begin{table}[h!]
    \centering\tiny
    \begin{tabular}{|c|c|c|c|c|c|}
    \hline
    \diagbox[linewidth=0.2pt, width=.85cm, height=.7cm]{$n$}{$p$} &2 &3 &4 &5 & 6 \\
    \hline
         8&$1: \mathbb{Z}^{}$ &&&& \\
         \hline
         
          \multirowcell{2}{9} &  $3: \mathbb{Z}^{}$ &&&&\\
        & $4: \mathbb{Z}^{3}$ &&&&\\
    \hline
         
          \multirowcell{2}{10} &\multirowcell{2}{$5: \mathbb{Z}^{16}$}  &$2: \mathbb{Z}^{}$&&&\\
        & & $3: \mathbb{Z}^{2}$&&&\\
    \hline
    
         11&$6: \mathbb{Z}^{43}$ & $4: \mathbb{Z}^{21}$&&& \\
         \hline
         
          \multirowcell{3}{12} &\multirowcell{3}{$7: \mathbb{Z}^{81}$}  & $5: \mathbb{Z}^{4}$&\multirowcell{2}{$3: \mathbb{Z}^{}$}&&\\
        & &$6: \mathbb{Z}^{2}$&\multirowcell{2}{$4: \mathbb{Z}^{3}$}&&\\
        && $7: \mathbb{Z}^{4}$&&&\\
    \hline
         
          \multirowcell{2}{13} &\multirowcell{2}{$8: \mathbb{Z}^{129}$}  &$7: \mathbb{Z}^{14}$ & $5: \mathbb{Z}^{27}$&&\\
        & &$8: \mathbb{Z}^{13}$& $6: \mathbb{Z}^{13}$&&\\
    \hline
         
          \multirowcell{2}{14} &\multirowcell{2}{$9: \mathbb{Z}^{188}$}  &$8: \mathbb{Z}^{15}$&\multirowcell{2}{$7: \mathbb{Z}^{99}$}&$4: \mathbb{Z}^{}$&\\
        & &$9: \mathbb{Z}^{35}$&&$5: \mathbb{Z}^2$&\\
    \hline
         
          \multirowcell{3}{15} &\multirowcell{3}{$10: \mathbb{Z}^{259}$}  & \multirowcell{2}{$9: \mathbb{Z}^{15}$}&\multirowcell{2}{$8: \mathbb{Z}^{77}$}&$6: \mathbb{Z}^{5}$&\\
        & &\multirowcell{2}{$10: \mathbb{Z}^{82}$}&\multirowcell{2}{$10: \mathbb{Z}^{5}$}&$7: \mathbb{Z}^{31}$&\\
        & &&&$8: \mathbb{Z}^{3}$&\\
    \hline
         
          \multirowcell{3}{16} &\multirowcell{3}{$11: \mathbb{Z}^{343}$}  &\multirowcell{2}{$10: \mathbb{Z}^{16}$}&$9: \mathbb{Z}^{11}$&\multirowcell{2}{$8: \mathbb{Z}^{181}$}&\multirowcell{2}{$5: \mathbb{Z}^{2}$}\\
        & &\multirowcell{2}{$11: \mathbb{Z}^{151}$}&$10: \mathbb{Z}^{8}$ &\multirowcell{2}{$9: \mathbb{Z}^{48}$} &\multirowcell{2}{$6: \mathbb{Z}^{}$}\\
        & && $11: \mathbb{Z}^{16}$&&\\
    \hline
         
          \multirowcell{2}{17} &\multirowcell{2}{$12: \mathbb{Z}^{441}$}  &$11: \mathbb{Z}^{17}$& $11: \mathbb{Z}^{52}$& $9: \mathbb{Z}^{18}$&$7: \mathbb{Z}^{}$\\
        & &$12: \mathbb{Z}^{237}$ &$12: \mathbb{Z}^{34}$& $10: \mathbb{Z}^{272}$&$8: \mathbb{Z}^{51}$\\
    \hline
         
          \multirowcell{3}{18} &\multirowcell{3}{$13:\mathbb{Z}^{554}$} & \multirowcell{2}{$12:\mathbb{Z}^{18}$}&\multirowcell{2}{$12:\mathbb{Z}^{73}$}&
        \multirowcell{2}{$11:\mathbb{Z}^{323}$} & $9:\mathbb{Z}^{16}$\\
        & &\multirowcell{2}{$13:\mathbb{Z}^{338}$}&\multirowcell{2}{$13:\mathbb{Z}^{69}$}& \multirowcell{2}{$13:\mathbb{Z}^{6}$}&$10:\mathbb{Z}^{159}$\\
        & &&&&$11:\mathbb{Z}^{18}$\\
    \hline
         
          \multirowcell{2}{19} &\multirowcell{2}{$14:\mathbb{Z}^{683}$}  &$13:\mathbb{Z}^{19}$&$13:\mathbb{Z}^{77}$& $12:\mathbb{Z}^{189}$ &$11:\mathbb{Z}^{723}$\\
        & &$14:\mathbb{Z}^{455}$ &$14:\mathbb{Z}^{133}$& $14:\mathbb{Z}^{19}$ &$12:\mathbb{Z}^{133}$\\
    \hline
    
          \multirowcell{3}{20} &\multirowcell{3}{$15:\mathbb{Z}^{829}$}  &\multirowcell{2}{$14:\mathbb{Z}^{20}$}& \multirowcell{2}{$14:\mathbb{Z}^{80}$} &$13:\mathbb{Z}^{28}$&\multirowcell{2}{$12:\mathbb{Z}^{186}$}\\
        & &\multirowcell{2}{$15:\mathbb{Z}^{589}$}& \multirowcell{2}{$15:\mathbb{Z}^{233}$} &$14:\mathbb{Z}^{20}$&\multirowcell{2}{$13:\mathbb{Z}^{604}$}\\
        & &&&$15:\mathbb{Z}^{40}$&\\
    \hline
    \end{tabular}
    \caption{Homology of $4$-cut complexes $\Delta_{4}(C_n^p)$}
    \label{tab:4-cut}
\end{table}

We have proved that for $n\geq 9=4\cdot2+1$, the $3$-cut complex $\Delta_3(C_n^2)$ is shellable and is homotopy equivalent to ${\binom{n-4}{2}-9}= \frac{n^2-8n-n+2}{2}$ wedge of spheres of dimension equal to the dimension of $\Delta_3(C_n^2)$. 
It is easy to check that for $p\geq 3$, if  $n \leq 2p+2$, then $\Delta_3(C_n^p)=\emptyset$.
Based on our calculation given in \Cref{tab:3-cut}, we make the following conjecture for $n \geq 2p+3$.
    
	\begin{conjecture}\label{conjecture:3-cut powered cycle}
		Let $p\geq 3$. Then $\Delta_3(C_n^p)$ satisfies the following.
        \begin{enumerate}[label=(\roman*)]
        \item For $2p+3\leq n \leq 4p$, $\Delta_3(C_n^p)$ is not shellable. Moreover,
        $$\Delta_3(C_{n}^p)\simeq\begin{cases}
                                    \mathbb{S}^{p-1},&\text{ if $n=2p+3$ and $p$ is even},\\
                    \mathbb{S}^{p},&\text{ if $n=2p+3$ and $p$ is odd,}\\
                    \bigvee\limits_{\alpha}\mathbb{S}^{n-8},&\text{ if $2p+4\leq n\leq 3p-1$,}\\
                    \bigvee\limits_{\beta}\mathbb{S}^{n-6},&\text{ if $3p\leq n\leq 4p-1$,}\\
                    \bigvee\limits_{p-1}\mathbb{S}^{n-5},&\text{ if $ n = 4p$,}
                                 \end{cases}$$
where $\alpha$ and $\beta$ depend on $n$ and $p$.

        \item\label{3-cut shellability} For $n\geq 4p+1$, $\Delta_3(C_n^p)$ is shellable and $$\Delta_3(C_n^p) \simeq \bigvee\limits_{\frac{n^2-4np-n+2}{2}} \mathbb{S}^{n-4}.$$ 
        \end{enumerate} 
		
	\end{conjecture}

\begin{table}[h!]
    \tiny
    \centering
    \begin{tabular}{|c|c|c|c|c|c|c|c|c|c|c|c|}
     \hline
    \diagbox[linewidth=0.2pt, width=.5cm, height=.5cm]{$p$}{$n$} &9 &10 &11 &12 &13 &14 &15 &16 &17 &18 &19\\
    \hline
    
      \multirowcell{2}{2} & \multirowcell{2}{$1: \mathbb{Z}^{}$} & $3: \mathbb{Z}^{}$& \multirowcell{2}{$5: \mathbb{Z}^{56}$}& \multirowcell{2}{$6: \mathbb{Z}^{152}$}& \multirowcell{2}{$7: \mathbb{Z}^{300}$}&\multirowcell{2}{$8: \mathbb{Z}^{505}$} & \multirowcell{2}{$9: \mathbb{Z}^{776}$} & \multirowcell{2}{$10: \mathbb{Z}^{1125}$} &\multirowcell{2}{$11: \mathbb{Z}^{1565}$} &\multirowcell{2}{$12: \mathbb{Z}^{2110}$} &\multirowcell{2}{$13: \mathbb{Z}^{2775}$}\\
      
       & & $4: \mathbb{Z}^{10}$&&&&&& &&&\\
    \hline
    
      \multirowcell{3}{3} & &&\multirowcell{3}{$3: \mathbb{Z}^{}$}& \multirowcell{3}{$4: \mathbb{Z}^{39}$}& $5: \mathbb{Z}^{15}$& \multirowcell{2}{$7: \mathbb{Z}^{41}$}& \multirowcell{2}{$8: \mathbb{Z}^{62}$} & \multirowcell{2}{$9: \mathbb{Z}^{64}$} &\multirowcell{2}{$10: \mathbb{Z}^{68}$}&\multirowcell{2}{$11: \mathbb{Z}^{72}$} &\multirowcell{2}{$12: \mathbb{Z}^{76}$}\\
      
       & &&& &$6: \mathbb{Z}^{}$&\multirowcell{2}{$8: \mathbb{Z}^{56}$}&\multirowcell{2}{$9: \mathbb{Z}^{165}$}&\multirowcell{2}{$10: \mathbb{Z}^{381}$} &\multirowcell{2}{$11: \mathbb{Z}^{732}$} &\multirowcell{2}{$12: \mathbb{Z}^{1213}$}&\multirowcell{2}{$13: \mathbb{Z}^{1825}$}\\
       
       & &&& &$7: \mathbb{Z}^{13}$&&& & &&\\
    \hline
    
      \multirowcell{3}{4} & &&&&\multirowcell{2}{$3: \mathbb{Z}^{2}$}& \multirowcell{3}{$5: \mathbb{Z}^{29}$}& \multirowcell{3}{$7: \mathbb{Z}^{228}$}& \multirowcell{3}{$8: \mathbb{Z}^{331}$}&$9: \mathbb{Z}^{155}$&$10: \mathbb{Z}^{23}$&$11: \mathbb{Z}^{19}$\\
      
       & &&&& \multirowcell{2}{$4: \mathbb{Z}^{}$}&&& &$10: \mathbb{Z}^{}$&$11: \mathbb{Z}^{163}$&$12: \mathbb{Z}^{360}$\\
       
       & &&& &&&& &$11: \mathbb{Z}^{68}$ &$12: \mathbb{Z}^{180}$&$13: \mathbb{Z}^{399}$\\
    \hline
    
      \multirowcell{2}{5} & &&&&&&\multirowcell{2}{$5: \mathbb{Z}^{}$}& $6: \mathbb{Z}^{68}$& \multirowcell{2}{$8: \mathbb{Z}^{271}$} &$9: \mathbb{Z}^{56}$& $11: \mathbb{Z}^{1426}$\\
      
       & &&& &&&&$7: \mathbb{Z}^{}$ & &$10: \mathbb{Z}^{606}$&$13: \mathbb{Z}^{19}$\\
    \hline
    
      \multirowcell{2}{6} & &&&&&&& &$5: \mathbb{Z}^{2}$ &$7: \mathbb{Z}^{8}$&{$9: \mathbb{Z}^{705}$}\\
      
       & &&& &&&& &6$: \mathbb{Z}^{}$ &$8: \mathbb{Z}^{78}$&{$10: \mathbb{Z}^{}$}\\
    \hline
    \end{tabular}
    \caption{Homology of $5$-cut complexes $\Delta_{5}(C_n^p)$}
    \label{tab:5-cut}
\end{table}

\begin{table}[h!]
    \tiny
    \centering
    \begin{tabular}{|c|c|c|c|c|c|c|c|c|c|c|}
     \hline
    \diagbox[linewidth=0.2pt, width=.5cm, height=.5cm]{$p$}{$n$}  &10 &11 &12 &13 &14 &15 &16 &17 &18 &19 \\
    \hline
    
      \multirowcell{2}{2} &\multirowcell{2}{$1: \mathbb{Z}^{}$} & $3: \mathbb{Z}^{}$& \multirowcell{2}{$5: \mathbb{Z}^{136}$}& \multirowcell{2}{$6: \mathbb{Z}^{402}$}& \multirowcell{2}{$7: \mathbb{Z}^{855}$} & \multirowcell{2}{$8: \mathbb{Z}^{1537}$} & \multirowcell{2}{$9: \mathbb{Z}^{2507}$} & \multirowcell{2}{$10: \mathbb{Z}^{3841}$} & \multirowcell{2}{$11: \mathbb{Z}^{5630}$} & \multirowcell{2}{$12: \mathbb{Z}^{7979}$}\\
      
       & & $4: \mathbb{Z}^{}$ &&&&&&&&\\
    \hline
      \multirowcell{3}{3} & &&\multirowcell{2}{$2: \mathbb{Z}^{}$}& \multirowcell{3}{$4: \mathbb{Z}^{64}$}& $5: \mathbb{Z}^{44}$& $6: \mathbb{Z}^{3}$& \multirowcell{2}{$8: \mathbb{Z}^{177}$} & \multirowcell{2}{$9: \mathbb{Z}^{205}$} & \multirowcell{2}{$10: \mathbb{Z}^{216}$} &\multirowcell{2}{ $11: \mathbb{Z}^{228}$}\\
      
       & & &\multirowcell{2}{$3: \mathbb{Z}^{2}$}&&$6: \mathbb{Z}^{}$& $7: \mathbb{Z}^{91}$& \multirowcell{2}{$9:\Z^{500}$} & \multirowcell{2}{$10: \mathbb{Z}^{1258}$}& \multirowcell{2}{$11: \mathbb{Z}^{2603}$}&\multirowcell{2}{$12:\Z^{4635}$}\\
      
       & & &&&$7: \mathbb{Z}^{28}$& $8: \mathbb{Z}^{150}$&&&&\\
    \hline
    
      \multirowcell{3}{4} & &&&&\multirowcell{2}{$3: \mathbb{Z}^{2}$}& \multirowcell{2}{$5: \mathbb{Z}^{58}$}& \multirowcell{2}{$6: \mathbb{Z}^{4}$}& \multirowcell{2}{$8: \mathbb{Z}^{917}$}& \multirowcell{2}{$9: \mathbb{Z}^{772}$}  & $10: \mathbb{Z}^{246}$ \\
      
       & & &&&\multirowcell{2}{$4: \mathbb{Z}^{}$}& \multirowcell{2}{$6: \mathbb{Z}^{5}$} & \multirowcell{2}{$7: \mathbb{Z}^{409}$}& \multirowcell{2}{$10: \mathbb{Z}^{34}$} & \multirowcell{2}{$11: \mathbb{Z}^{180}$} & $11: \mathbb{Z}^{342}$\\

       &&&&&&&&&&$12: \mathbb{Z}^{570}$\\
    \hline
    
      \multirowcell{3}{5} & &&&&&&\multirowcell{2}{$4: \mathbb{Z}^{4}$}& \multirowcell{3}{$6: \mathbb{Z}^{101}$}& $7: \mathbb{Z}^{2}$ & \multirowcell{2}{$9: \mathbb{Z}^{476}$ }\\
      
       & & &&&&&\multirowcell{2}{$5: \mathbb{Z}^{}$}&&$8: \mathbb{Z}^{451}$&\multirowcell{2}{$10: \mathbb{Z}^{1159}$ }\\

       & & &&&&&&&$9: \mathbb{Z}^{9}$&\\
    \hline
      \multirowcell{2}{6} & &&&&&& & & $5: \mathbb{Z}^{2}$ & $7: \mathbb{Z}^{135}$ \\
      
       & & &&&&&&&$6: \mathbb{Z}^{}$&$8: \mathbb{Z}^{20}$\\
    \hline
    \end{tabular}
    \caption{Homology of $6$-cut complexes $\Delta_{6}(C_n^p)$}
    \label{tab:6-cut}
\end{table}

 Examining Tables \ref{tab:4-cut}-\ref{tab:6-cut} for $k=4,5,6$, we observe that for $p\geq 3$, the complexes $\Delta_k(C_n^p)$  have non trivial homology in dimensions lower than the  dimension of $\Delta_k(C_n^p)$. These findings raise the question of whether the complexes $\Delta_k(C_n^p)$, if nonvoid, are not shellable for $k\geq 4$ and $p\geq 3$.

    \Cref{conjecture} focuses on the shellability of the $k$-cut complex $\Delta_k(C_n^2)$ of the squared cycle graphs for $k\geq 3$. In \Cref{theorem:main}, we have proved that for $n\geq 9=2\cdot3+3$, $\Delta_3(C_n^2)$ is homotopy equivalent to wedge of ${\binom{n-4}{2}-9}= \binom{n-1}{3-1}-(2^{3-1}-1)n$ spheres of dimension $n-4$. Using the data in Tables \ref{tab:4-cut}-\ref{tab:6-cut}, we extend this discussion with the following conjecture about the homotopy type of $\Delta_k(C_n^2)$ for $k\geq 4$.
    
    \begin{conjecture}\label{conjecture:k-cut squared cycle}  For $k\geq4$ and $n\geq2k+3$, 
    $$
		\Delta_k(C_n^2) \simeq \bigvee\limits_{{n-1\choose k-1}-(2^{k-1}-1)n} \mathbb{S}^{n-k-1}.
		$$
	\end{conjecture}

\section*{Acknowledgement}
The first author is supported by HTRA fellowship by IIT Mandi, India. The second author is supported by the seed grant project IITM/SG/SMS/95 by IIT Mandi, India. The third author is partially supported by the DISHA fellowship by NISER Bhubaneswar, India.

\bibliographystyle{abbrv}
\bibliography{ref}

\addcontentsline{toc}{section}{References}

\end{document}